\documentclass[a4paper,11pt,reqno]{amsart}

\usepackage{amssymb}
\usepackage{epsfig}
\usepackage{amsfonts}
\usepackage{mathrsfs}
\usepackage{amsmath}
\usepackage{euscript}
\usepackage{amscd}
\usepackage{amsthm}
\usepackage{enumitem}
\DeclareMathAlphabet{\mathpzc}{OT1}{pzc}{m}{it}
\usepackage{enumitem}
\usepackage{color}
\usepackage{mathtools}

\DeclareSymbolFont{SY}{U}{psy}{m}{n}
\DeclareMathSymbol{\emptyset}{\mathord}{SY}{'306}

\theoremstyle{plain}

\newtheorem{thm}{Theorem}[section]
\newtheorem{cor}[thm]{Corollary}
\newtheorem{lem}[thm]{Lemma}
\newtheorem{prop}[thm]{Proposition}
\theoremstyle{definition}
\newtheorem{defn}[thm]{Definition}
\newtheorem{rem}[thm]{Remark}
\newtheorem{ex}[thm]{Example}

\numberwithin{equation}{section}

\def\C{{\mathbb C}}

\def\norm#1{\left\|{#1}\right\|}

\def\N{\mathbb{N}}

\def\v{\varphi}
\def\l{\lambda}

\def\ra{\rightarrow}
\def\ov{\overline}
\def\lo{\longrightarrow}

\def\m{\mathcal}
\def\mb{\mathbb}

\def\a{\alpha}

\def\g{\gamma}

\def\wi{\widetilde}

\def\beq{\begin{eqnarray}}
\def\eeq{\end{eqnarray}}
\def\beqa{\begin{eqnarray*}}
\def\eeqa{\end{eqnarray*}}

\def\ov{\overline}
\def\bl{\boldsymbol}

\def\i{\prime}

\newcommand{\be}{\begin{equation}}
\newcommand{\ee}{\end{equation}}
\newcommand{\bea}{\begin{eqnarray}}
\newcommand{\eea}{\end{eqnarray}}
\newcommand{\Bea}{\begin{eqnarray*}}
\newcommand{\Eea}{\end{eqnarray*}}
\newcommand{\inner}[2]{\langle #1,#2 \rangle }%

\newcounter{cnt1}
\newcounter{cnt2}
\newcounter{cnt3}
\newcommand{\blr}{\begin{list}{$($\roman{cnt1}$)$}
 {\usecounter{cnt1} \setlength{\topsep}{0pt}
 \setlength{\itemsep}{0pt}}}
\newcommand{\bla}{\begin{list}{$($\alph{cnt2}$)$}
 {\usecounter{cnt2} \setlength{\topsep}{0pt}
 \setlength{\itemsep}{0pt}}}
\newcommand{\bln}{\begin{list}{$($\arabic{cnt3}$)$}
 {\usecounter{cnt3} \setlength{\topsep}{0pt}
 \setlength{\itemsep}{0pt}}}
\newcommand{\el}{\end{list}}

\begin{document}
\title[Contractive Hilbert modules on  quotient domains]{Contractive Hilbert modules on  quotient domains}
\author[Biswas]{Shibananda Biswas}
\address[Biswas]{Indian Institute of Science Education and Research Kolkata, 741246, Nadia, West Bengal, India}
\email{shibananda@iiserkol.com}
\author[Ghosh]{Gargi Ghosh$^\dag$}
\address[Ghosh]{Jagiellonian University, 30-348 Krakow, Poland    \newline      Current Address: Indian Institute of Technology Roorkee, 247667, India}\email{gargi.ghosh@ma.iitr.ac.in}
\author[Narayanan]{E. K. Narayanan}\address[Narayanan]{Indian Institute of Science, Bangalore, 560012, India} \email{naru@iisc.ac.in}
\author[Shyam Roy]{Subrata Shyam Roy}
\address[Shyam Roy]{Indian Institute of Science Education and Research Kolkata, 741246, Nadia, West Bengal, India}
\email{ssroy@iiserkol.ac.in}
\thanks{$^\dag$This work is supported by the project No. 2022/45/P/ST1/01028 co-funded by the
National Science Centre and the European Union Framework Programme for Research and Innovation Horizon 2020 under the Marie Sklodowska-Curie grant agreement No. 945339. For the purpose of Open Access, the author has applied a CC-BY public
copyright licence to any Author Accepted Manuscript (AAM) version arising from this
submission.}

\begin{abstract}
Let the complex reflection group $G(m,p,n)$ act on the unit polydisc $\mb D^n$ in $\mb C^n.$  A $\boldsymbol\Theta_n$-{\it contraction} is a commuting tuple of operators on a Hilbert space having 
$$\ov{\boldsymbol\Theta}_n:=\{\boldsymbol\theta(z)=(\theta_1(z),\ldots,\theta_n(z)):z\in\ov{\mb D}^n\}$$  as a  spectral set, where $\{\theta_i\}_{i=1}^n$  is a homogeneous system of parameters associated to $G(m,p,n).$ A plethora of examples of $\boldsymbol\Theta_n$-contractions is exhibited. Under a mild hypothesis, it is shown that these $\boldsymbol\Theta_n$-contractions are mutually unitarily inequivalent.  These inequivalence results are obtained concretely for the weighted Bergman modules under the action of the permutation groups and the dihedral groups.  The division problem is shown to have  negative answers for the Hardy module and the Bergman module on the bidisc. A Beurling-Lax-Halmos type representation for the invariant subspaces of $\boldsymbol\Theta_n$-isometries is obtained.
\end{abstract}

\maketitle
\section{Introduction} 
The notion of a quotient domain $\Omega/G$  for a bounded domain $\Omega\subseteq\mb C^n$ by a finite complex reflection group  $G$  emerged in \cite{BDGS}. After that there is a flurry of activity involving this notion, for instance, see \cite{CKY, DM, Ghosh, GG, GN, GS, GZ, HW}. A {\it complex reflection} on $\mb C^n$ is a linear mapping $\sigma:\mb C^n\to\mb C^n$ such that $\sigma$ has finite order in $GL(n,\C)$ and the rank of $I_n-\sigma$ is $1.$ A group $G$ generated by complex reflections is called a {\it complex reflection group}. 
 Let $G$ be a finite complex reflection group which acts (right action) on $\mb D^n$ by \begin{eqnarray}\label{action} \sigma \cdot \bl z =  {\sigma}^{-1}\bl z, \,\, \text{~for~} \sigma \in G \text{~and~} \bl z \in \mb D^n.\end{eqnarray} The group action extends to the set of all complex-valued functions on $\mb D^n$ by $\sigma( f)(\bl z) =  f({\sigma}^{-1}\cdot \bl z)$ and a function $f$ is said to be $G$-{\it invariant} if $\sigma(f)=f$ for all $\sigma \in G.$ There is a system of $G$-invariant  algebraically independent homogeneous polynomials $\{\theta_i\}_{i=1}^n$ associated to a complex reflection group $G,$ called a homogeneous system of  parameters (hsop) or basic polynomials associated to $G.$  
In fact, Chevalley-Shephard-Todd theorem provides a characterization of finite complex reflection groups in terms of hsop. It states that a finite group $G$ is generated by complex reflections if and only if the ring of $G$-invariant polynomials 
$$\mb C[z_1,\ldots,z_n]^G=\{f\in\mb C[z_1,\ldots, z_n]:\sigma(f)=f \text{~for all~} \sigma\in G\}$$
forms a polynomial ring $\mb C[\theta_1,\ldots,\theta_n]$ \cite[p.282]{Shephard-Todd}. The polynomial map \begin{eqnarray*} \bl \theta =(\theta_1,\ldots,\theta_n) : \mb D^n \to \bl \theta(\mb D^n) \end{eqnarray*} is a proper holomorphic map and the domain $\bl \theta(\mb D^n)$ is biholomorphically equivalent to
 the quotient $\mb D^n/G$ \cite{trybula, BDGS}. So we refer to the domains of the form $\bl \theta(\mb D^n)$ by \emph{quotient domains.}

The irreducible finite complex reflection groups were classified by Shephard and Todd in \cite{Shephard-Todd}. They proved that every irreducible complex reflection group belongs to an infinite family $G(m, p, n)$ indexed by three parameters, where $ m,n,p$ are  positive integers and $p$ divides $m,$  or, is one of 34 exceptional groups. Although for certain values of $m,p$ and $n,$ $G(m,p,n)$ can be reducible, for example, $G(2,2,2)$ is the dihedral group of order $4$ which is isomorphic to the product of two groups of order $2.$  We note that $G(m,p,n)=$ $n\times n$ monomial matrices whose nonzero entries are $m$-th root of unity with product a $\frac{m}{p}$-th root of unity, where $m,p,n$ are positive integers and $p$ divides $m.$ It is known that $G(m,p,n)$ contains complex reflections of two types when $p<m$: for $\xi_m=e^{2\pi i/m}$ we have \begin{enumerate}
    \item transposition like reflections $\sigma_{ij|k}$ for $1\leq i<j\leq n$ and $k=1,\ldots,m$ such that the $i$-th component and $j$-th component of $\sigma_{ij|k} \cdot (z_1,\ldots,z_n)$ are $\xi_m^{k}z_i$ and $\xi_m^{-k}z_j,$ respectively and any other component remains the same.
    \item The diagonal reflection $\sigma_{i|k}$ for $1\leq i \leq n$ and $k=p,2p,\ldots,(\frac{m}{p}-1)p$ acts by $\sigma_{i|k} \cdot (z_1,\ldots,z_n)=(z_1,\ldots,z_{i-1},\xi_m^kz_i,z_{i+1},\ldots,z_n).$
\end{enumerate} The group $G(m,m,n)$ contains only transposition like reflections. In other words, the above discussion provides an explicit description of the action of $G(m,p,n)$ on $\mb D^n$. For a detailed study on $G(m, p, n),$ we refer to \cite[Chapter 2]{LT09}.

 Let $\bl s$ denote the symmetrization map $\bl s=(s_1,\ldots,s_n):\mb C^n\to\mb C^n$ defined by 
$$
s_i( z) = \sum_{1\leq k_1< k_2 <\ldots <k_i\leq n} z_{k_1} \cdots z_{k_i}  \mbox{~for~} i= 1,\ldots,n.
$$
For $n>1,$ a set of basic invariant polynomials for the group $G(m,p,n)$ is given by elementary symmetric polynomials of $z_1^m,\ldots,z_n^m$ of degrees $1,\ldots,n-1$ and $(z_1\cdots z_n)^q,$ where $q=m/p$ \cite[p. 36]{LT09}. We denote the elementary symmetric polynomials of degree $i$ of $z_1^m,\ldots,z_n^m$ by $\theta_i(z)$ for $i=1,\ldots,n-1$ and $\theta_n(z) = (z_1\cdots z_n)^q.$ That is, 
\Bea
\theta_i(z)=s_i(z_1^m,\ldots,z_n^m) \,\,\text{~for~} 1\leq i\leq n-1 \text{~and~}
\theta_n(z_1,\ldots,z_n)=s_n(z_1,\ldots,z_n)^q=(z_1\ldots z_n)^q. 
\Eea
Note that $\{\theta_i\}_{i=1}^n$ satisfies 
\bea\nonumber
\ov{\theta_i(z)}\theta_n^p(z)=\theta_{n-i}(z) \,\, \text{~for~} z\in\mb T^n,
\eea
$\mb T$ being the unit circle in the complex plane.

The group $G(m, p, n)$ has an action on $\mb D^n$ as in Equation \eqref{action}. Thus we have an explicit description for the basic polynomial map $\bl \theta:=(\theta_1,\ldots,\theta_n) : \mb D^n \to \bl\theta(\mb D^n)$ associated to $G(m,p,n).$ Any image of $\mb D^n$ under the proper holomorphic mapping $\bl \pi$ with $\operatorname{Deck}({\bl \pi})=G(m,p,n),\,n>1,$ is biholomorphic to the associated $\bl \theta(\mb D^n).$ Put $\bl\Theta_n=\bl\theta(\mb D^n).$

Any  commuting tuple of operators on a Hilbert space having $\ov{\bl\Theta}_n$ as a spectral set will be called a $\bl\Theta_n$-{\it contraction}. We say that $\ov{\bl\Theta}_n$ is a {\it spectral set} for a commuting tuple $(T_1,\ldots, T_n)$ of operators on a Hilbert space if, for every polynomial $f\in\mb C[z_1,\ldots,z_n]$
 \bea\label{spec}
 \Vert f(T_1,\ldots, T_n)\Vert\leq\sup_{\ov{\bl\Theta}_n}\vert f\vert=\Vert f\Vert_{\ov{\bl\Theta}_n,\infty}.
 \eea 

For the particular case of $G(1,1,n),\,n>1,$ that is, the permutation group $\mathfrak S_n$ on $n$ symbols, $\ov{\bl\Theta}_n$ coincides with  $\Gamma_n,$ the closed symmetrized polydisc. An exception in terminology has been made here by  calling a commuting tuple of operators having $\ov{\bl\Theta}_n$ as a spectral set $\bl\Theta_n$-contraction rather than a $\ov{\bl\Theta}_n$-contraction.  The study of $\Gamma_n$-contractions ($n>2$) was initiated in \cite{BS} along the lines of \cite{AY}. A large class of examples of $\Gamma_n$-contractions arising from weighted Bergman modules  was exhibited in \cite{BS}. The question of mutual unitary equivalence of these $\Gamma_n$-contractions was studied in \cite{BGMS} with a number of interesting results.

Since the  map $\bl\theta:\mb C^n\to\mb C^n$ is a proper holomorphic map and $\ov{\mb D}^n$ is polynomially convex,  $\bl\theta(\ov{\mb D}^n)=\ov{\bl\Theta}_n$ is polynomially convex by
\cite[Theorem 1.6.24]{ELS}. 
Hence the Taylor spectrum of a $\bl\Theta_n$-contraction is contained in $\ov{\bl\Theta}_n.$ Although $\bl\Theta_n$-contractions are defined  by requiring that the inequality \eqref{spec} holds for all polynomials $f$ in $\mb C[z_1,\ldots,z_n],$ this is equivalent to the definition of $\bl\Theta_n$-contractions by requiring \eqref{spec} to hold for all functions $f$ analytic in a neighbourhood of $\ov{\bl\Theta}_n$ due to polynomial convexity of $\ov{\bl\Theta}_n$ and Oka-Weil theorem. As explained in \cite{AY}, the subtleties involving the various notions of joint spectrum and functional calculus for commuting tuples of operators  are not relevant for $\bl\Theta_n$-contractions, simply because of the polynomial convexity of $\bl\Theta_n.$

 In this article, we study properties of $\bl\Theta_n$-contractions and find a number of ways of constructing $\bl\Theta_n$-contractions. We obtain a model for $\bl\Theta_n$-isometries. As an application,  a Beurling-Lax-Halmos type theorem characterizing joint invariant subspaces of a pure $\bl\Theta_n$-isometries is established. 

 In the concluding section, we exhibit a number of different classes of $\bl\Theta_n$-contractions  applying the results developed in the preceding sections. For instance, Theorem \ref{reduced} exhibits a plethora of examples of $\bl\Theta_n$-contractions and Corollary \ref{number} shows under a mild hypothesis that there are precisely $\vert\widehat G\vert$ number of mutually (unitarily) inequivalent  $\bl\Theta_n$-contractions  arises in Theorem \ref{reduced}, $\widehat G$ being the set of all equivalence classes of irreducible representations of the group $G.$

 Theorem \ref{Theta} is  a concrete realization of Corollary \ref{number} in the case of weighted Bergman modules $\mb A^{(\l)}(\mb D^n),\,\,\l\geq 1.$ Moreover, the following inequivalence results are obtained for the weighted  Bergman modules $\mb A^{(\l)}(\mb D^n),\,\,\l\geq 1$
\begin{enumerate}
    \item [(i)] For the group $G(1,1,n),\,\,n>1,$ it is shown in  Theorem \ref{Gamma} that the family of $\Gamma_n$-contractions indexed by $\widehat{G(1,1,n)}$ are unitarily inequivalent. Moreover, it is shown in Theorem \ref{similar} that they are not even similar. 
    \item[(ii)] For the group dihedral group $G(k,k,2),\,\,k>1$ Theorem \ref{Gamma} shows that the family of $\mathscr D_{2k}$-contractions indexed by $\widehat{G(k,k,2)}$ are unitarily inequivalent.
\end{enumerate}
In addition, as a consequence Proposition \ref{unbdd}, it follows that the division problem is shown to have negative answers for the Hardy module and the Bergman module on the bidisc (see Remark \ref{div}).


\section{$\bl\Theta_n$ and $\bl\Theta_n$-contractions}
 Let $\bl s$ denote the symmetrization map $\bl s=(s_1,\ldots,s_n):\mb C^n\to\mb C^n$ defined by 
$$
s_i( z) = \sum_{1\leq k_1< k_2 <\ldots <k_i\leq n} z_{k_1} \cdots z_{k_i}  \mbox{~for~} i= 1,\ldots,n.
$$
Set $s_0:=1.$ Then $\theta_0:=s_0=1,$ and $\theta_i:\mb C^n\to \mb C$ is defined by  \Bea\theta_i(z)&=&s_i(z_1^m,\ldots,z_n^m) \text{ for } 1\leq i\leq n-1\text{ and }\\  \theta_n^p(z)&=&z_1^m\cdots z_n^m. \Eea  Thus $(\theta_1,\ldots,\theta_n)\in {\ov{\bl\Theta}}_n$ if and only if all the zeros of the polynomial 
$$\sum_{i=0}^{n - 1} (-1)^{i}\theta_{i}z^{n - i}+(-1)^n\theta_n^p$$
lie in $\ov{\mb D}.$ This realization of points of ${\ov{\bl\Theta}}_n$  will be used repeatedly.

Let $\mb C[z_1,\ldots,z_n]$ denote the algebra of all complex polynomials in $n$ variables. We also use the notation $\mb C[\bl z]$ instead of $\mb C[z_1,\ldots,z_n]$ if there is no ambiguity. A polynomial $f\in \mb C[z_1,\ldots,z_n]$ is in called $G$-invariant if 
$$(f \circ \sigma) (\bl z) = f({\sigma^{-1}} \cdot \bl z):=f(\bl z) \text{~for all~} \sigma\in G  \text{~and~} \bl z\in \mb C^n.$$ 
The ring of $G$-invariant polynomials in $n$ variables is denoted by $\mb C[z_1,\ldots,z_n]^G = \{f\in \mb C[\bl z] : f \circ \sigma = f \text{ for all } \sigma \in G\}.$ If $f\in \mb C[z_1,\ldots,z_n]$ is $G$-invariant, then there is a unique $q\in \mb C[z_1,\ldots,z_n]$ such that $f=q\circ \bl\theta,$ where $\bl\theta$ is a basic polynomial mapping associated to $G(m,p,n)$ \cite[Theorem 3.3.1]{E}. 
 
We state two classical theorems on location of zeros of polynomials which will be useful in the sequel. For a polynomial $f\in\mb C[z],$ the derivative of $p$ with respect to $z$ will be denoted by $f^\prime.$

\begin{thm}\label{GL}(Gauss-Lucas, \cite[p. 22]{M}) The zeros of the derivative of a polynomial $f$ lie in the convex hull of the
zeros of $f.$
\end{thm}
\noindent We recall a definition.
\begin{defn}
A polynomial $f\in\mb C[z]$ of degree $d$ is called \emph{self-inversive} if  $z^d\overline{f(\frac{1}{\bar z})}=\omega f(z)$
for some constant $\omega\in \mb C$ with $\vert \omega\vert=1.$
\end{defn}
\begin{thm} \label{Cohn}(Cohn, \cite[p. 206]{M})
 A necessary and sufficient condition for all the zeros of a polynomial $p$ to lie on the unit circle $\mb T$ is that
$f$ is self-inversive and all the zeros of $f^\prime$ lie in the closed unit disc $\ov{\mb D}.$
\end{thm}

\begin{defn}\cite{Gamelin}
The Shilov boundary $\partial \Omega$ of a bounded domain $\Omega$ is given by the
closure of the set of its peak points and a point $\bl w \in \ov{\Omega}$ is said to be a peak point of $\Omega$ if there exists a function $f \in \m A(\Omega)$ such that $|f(\bl w)|>|f(\bl z)|$ for all $\bl z \in \ov{\Omega} \setminus \{\bl w\},$ where $\m A(\Omega)$ denotes the algebra of all functions holomorphic on $\Omega$ and continuous on $\ov{\Omega}.$
\end{defn} Since the distinguished boundary of $\ov{\Omega}$ in $\mb C^n$ is the Shilov boundary of $\Omega,$ these two notions will be used interchangeably without any confusion. The proper holomorphic map $\bl \theta : \mb D^n \to \bl \theta(\mb D^n)$ can be extended to a proper holomorphic map of the same multiplicity from $D'$ to $\bl \theta(D)',$ where the open sets $D'$ and $\bl \theta(D)'$ contain $\overline{\mb D^n}$ and $\overline{\bl \theta(\mb D^n)},$ respectively. Then \cite[p. 100, Corollary 3.2]{kosinski-zwonek} states that $\bl \theta^{-1}(\partial\bl \theta(\mb D^n)) = \partial\mb D^n=\mb T^n.$ Thus \bea\label{shilovboundary}\partial\bl \theta(\mb D^n) = \bl \theta(\mb T^n).\eea
We shall need the following characterizations of the Shilov boundary of $\bl\Theta_n$.
\begin{thm}\label{bg}
 Let $\theta_i\in\mathbb C, i=1,\ldots,n.$ Then the following are equivalent:
 \begin{enumerate}
  \item[(i)] $(\theta_1,\ldots, \theta_{n})$ is in the Shilov boundary of $\bl\Theta_n.$
  \item[(ii)]$\vert \theta_n\vert=1, \theta_n^p \ov \theta_j=\theta_{n-j}$ and $(\gamma_1\theta_1,\ldots,\gamma_{n-1}\theta_{n-1})\in\Gamma_{n-1},$
for $j=1,\ldots,n-1,$ where $\gamma_j=\frac{n-j}{n}$; 
  \item[(iii)] $(\theta_1,\ldots, \theta_n)\in \ov{\bl\Theta}_n$ and $\vert \theta_n\vert=1$;
  \item[(iv)] there exists  $(\mu_1,\ldots, \mu_{n - 1})$  in the Shilov boundary of $\bl\Theta_{n - 1}$ such that $\theta_j = \mu_j + {\ov\mu}_{n - j}\theta_n^p$ for $ j = 1,\ldots, n - 1.$
 \end{enumerate}
\end{thm}

\begin{proof}
 Throughout this proof, we put $\theta_0=1.$ Let $(\theta_1,\ldots, \theta_n)$ be in the Shilov boundary of $\bl\Theta_n.$
By definition, there are $\lambda_j\in \mathbb T, j=1,\ldots, n,$  such that
 \beq\label{elemen}
  \theta_j=\sum_{1\leq k_1<\ldots<k_j\leq n}\lambda_{k_1}^m\ldots \lambda_{k_j}^m\mbox{~for~} j=1,\ldots,n-1 \text{ and }  \theta_n^p =\lambda_{1}^m\ldots \lambda_{n}^m.
 \eeq
 This is equivalent to the fact that the polynomial $f$ given
 by
\beq\label{p}
f(z)=\sum_{i=0}^{n - 1} (-1)^{i}\theta_{i}z^{n - i}+(-1)^n\theta_n^p
\eeq
 has all its zeros on $\mathbb T.$ Moreover,  $\vert \theta_n\vert=1$ and $  \theta_n^p \ov\theta_j= \theta_{n-j}$  for $j=1,\ldots,n-1.$ It follows from Theorem \ref{Cohn}
 that the polynomial
\beq\label{dp}
f^\prime(z)=\sum_{j=0}^{n - 1}(-1)^{j}(n - j)\theta_{j}z^{n-j-1}
\eeq
has all its roots in $\overline{\mathbb D},$
 which is equivalent to the fact that $(\gamma_1
 \theta_1,\ldots,\gamma_{n-1}\theta_{n-1})\in\Gamma_{n-1},$ where
 $\gamma_j=\frac{n-j}{n}$ for $j=1,\ldots,n-1.$ Therefore (i) implies (ii).

 Conversely, considering the polynomial in Equation \eqref{p}, we observe that
 \beqa
z^n\overline{f\bigg(\frac{1}{\ov z}\bigg)}=\sum_{j=0}^{n-1}(-1)^j\ov \theta_jz^j +(-1)^n{\ov\theta_n}^pz^n\mbox{~ and~}
 z^n\overline{f\bigg(\frac{1}{\ov z}\bigg)}=(-1)^n{\ov\theta_n}^pf(z)
\eeqa
 by the first part of (ii). Therefore, $f$ is a self-inversive polynomial. Note that all the roots of $f^{\prime}$ lies in $\overline{\mathbb D}$ as $(\gamma_1\theta_1,\ldots,\gamma_{n-1}\theta_{n-1})\in\Gamma_{n-1},$  where
 $\gamma_j=\frac{n-j}{n}$ for $i=1,\ldots,n-1.$
 Thus, it follows from  Theorem \ref{Cohn}
 that $f$ has all  its roots on $\mathbb T.$ This is same as saying that $(\theta_1,\ldots,\theta_n)$ is in the Shilov boundary of $\bl\Theta_n.$

 Clearly, (i) implies (iii). To see the converse, we note that (iii) implies that there exist $\l_j\in\overline{\mb D}$ for $j=1,\ldots,n,$
 such that Equation \eqref{elemen} holds and $\vert \theta_n^p\vert=\vert\l_1^m\vert\ldots\vert\l_n^m\vert=1.$ So $ \l_j\in \mb T$ for
all $j=1,\ldots,n.$ It follows from the definition and Equation \eqref{elemen} that $(\theta_1,\ldots,\theta_n)$ is in the Shilov boundary of $\bl\Theta_n.$
 
 To show (i) implies (iv), let $(\theta_1,\ldots, \theta_{n})\in \partial\bl \Theta_n$. Clearly, $\theta_n^p =
 \l_1^m\ldots\l_n^m= e^{it}$ for some $t$ in $\mathbb R.$ So $\l_n^m=e^{it}\ov\l_1^m\ldots\bar\l_{n-1}^m.$ Since
 $$ \theta_j=\sum_{1\leq k_1< k_2 <\ldots <k_j\leq n} \l_{k_1}^m \cdots \l_{k_j}^m\mbox{~~with ~~}\vert \l_j\vert=1$$
 for $j=1,\ldots,n-1,$ we have
 \beqa
  \theta_j=\sum_{1\leq k_1< k_2 <\ldots <k_j\leq n} \l^m_{k_1} \cdots \l^m_{k_j}=\mu_j+\mu_{j-1} \l_n^m
  \eeqa
where $$\mu_j=\sum_{1\leq k_1< k_2 <\ldots <k_j\leq n-1} \l_{k_1}^m \cdots \l_{k_j}^m.$$
Thus $(\mu_1,\ldots,\mu_{n-1})\in \partial \bl\Theta_{n-1}$. Since $\l_n^m=e^{it}\ov\l_1^m\ldots\ov\l_{n-1}^m=\ov\mu_{n-1}e^{it}=\ov\mu_{n-1}\theta_n^p,$ we obtain
\beqa
\theta_j=\mu_j+\mu_{j-1}e^{it}\ov\l_1^m\ldots\ov\l_{n-1}^m=\mu_j+\mu_{j-1}\ov \mu_{n-1}e^{it}
  = \mu_j+\ov\mu_{n-j} \theta_n^p,
 \eeqa
 as $\ov\mu_{n-1}\mu_{j-1}=\ov\mu_{n-j},\ j=1,\ldots, n-1.$ Conversely, let $\mu_j = \theta_j(\bl\lambda)$ for some $\bl\lambda = (\lambda_1,\ldots,\lambda_{n - 1})\in\mb T^{n - 1}$. We choose $\lambda_n^m = {\ov\mu}_{n - 1}\theta_n^p\in\mb T$. Thus $\theta_j = \mu_j + \bar\mu_{n - j} \l_n^m\mu_{n - 1} = \mu_j + \mu_{j- 1}\l_n^m = \theta_j(\lambda_1,...,\lambda_n)$ and $\theta_n = \theta_n(\lambda_1,...,\lambda_n)$ and hence $(\theta_1,\ldots, \theta_{n})$ is in $\partial\bl\Theta_n$. Hence (iv) implies (i).
\end{proof}
 
 \begin{lem}\label{proj}
 Consider the map $\pi:\mb C^n\to\mb C^{n-1}$ defined by
 $$\pi(z_1,\ldots,z_n)=(\g_1z_1,\ldots,\g_{n-1}z_{n-1}) \text{~for~} \gamma_i=\frac{n-i}{n},\,\,i=1,\ldots,n-1.$$
 Then $\pi(\ov{\bl\Theta}_n)\subseteq\Gamma_{n-1}.$
   \end{lem}
\begin{proof}
Since $(\theta_1,\ldots, \theta_n) \in \ov{\bl\Theta}_n,$  the polynomial $f(z)=\sum_{j=1}^n (-1)^{n-j}\theta_{n-j}z^i+(-1)^n\theta_n^p,$ has all its roots in the closed unit disc $\overline{\mb D}.$ It follows from Theorem \ref{GL} that the polynomial $$f^\prime(z)=\sum_{j=0}^{n - 1}(-1)^{j}(n - j)\theta_{j}z^{n-j-1}$$ 
has all its roots
in $\ov{\mb D}$ as well. Hence we have the desired conclusion.
\end{proof}

We have an operator analogue of the previous Lemma.

\begin{lem}\label{projection}
 If $(T_1,\ldots, T_n)$ is a ${\bl\Theta}_n$-contraction, then $(\g_1T_1,\ldots,\g_{n-1}T_{n-1})$ is a $\Gamma_{n-1}$-contraction,
   where $\g_j=\frac{n-j}{n}$ for $j=1,\ldots,n-1.$
\end{lem}
\begin{proof}
 For $f\in\mb C[z_1,\ldots z_{n-1}],$ we note that $f\circ \pi\in\mb C[z_1,\ldots,z_n]$ and by hypothesis
 \beqa
 \Vert f(\g_1T_1,\ldots,\g_{n-1}T_{n-1})\Vert&=&\Vert( f\circ\pi)(T_1,\ldots,T_n)\Vert\\
 &\leq& \Vert  f\circ\pi \Vert_{\infty,\ov{\bl\Theta}_n}\\
 &= & \Vert  f\Vert_{\infty,\pi(\ov{\bl\Theta}_n)}\leq \Vert  f\Vert_{\infty,\Gamma_{n-1}}.\\
 \eeqa This completes the proof.
\end{proof}


For a commuting tuple $\mathbf T=(T_1,\ldots,T_n)$ of operators on a Hilbert space $\m H,$ let $\bl\theta(\mathbf T)
:=(\theta_1(\mathbf T),\ldots,\theta_n(\mathbf T)).$ 

\begin{prop}\label{vN}
Let $(T_1,\ldots,T_n)$ be an tuple of commuting contractions.  Then $\bl\theta(\mathbf T)$  is a ${\bl\Theta}_n$-contraction if and only if $(T_1,\ldots,T_n)$ satisfies   von Neumann's inequality for all $G$-invariant polynomials in $\mb C[z_1,\ldots,z_n].$
\end{prop}
\begin{proof}
 Let $\mathbf T=(T_1,\ldots,T_n)$ be a commuting tuple of contractions satisfying the analogue of von Neumann's
inequality for all $G$-invariant polynomials in $\mb C[z_1,\ldots,z_n].$
 
For a polynomial $f\in\mb C[z_1,\ldots,z_n],$ we note that
\beqa
\Vert f(\bl\theta(\mathbf T))\Vert= \Vert (f\circ \bl\theta)( \mathbf T)\Vert
\leq \Vert f\circ \bl\theta\Vert_{\infty,\ov{\mb D}^n}
= \Vert f\Vert_{\infty,\bl\theta(\ov{\mb D}^n)},
\eeqa
since $f\circ\bl\theta $ is a $G$-invariant polynomial, the above inequality holds by hypothesis, and it shows that
$\ov{\bl\Theta}_n=\bl\theta(\ov{\mb D}^n)$ is a spectral set for $\bl\theta(\mathbf T),$ thus $\bl\theta(\mathbf T)$ is a ${\bl\Theta_n}$-contraction.

Conversely, let $\mathbf T=(T_1,\ldots,T_n)$ be  a tuple of commuting contractions such that $\bl \theta(\mathbf T)$ is a
${\bl\Theta}_n$-contraction and  $q\in\mb C[z_1,\ldots,z_n]$ be $G$-invariant. So there is a $f\in\mb C[z_1,\ldots,z_n]$ such that
$f\circ \bl \theta=q.$ By hypothesis, we have
\beqa
\Vert q(\mathbf T)\Vert=\Vert f( \bl \theta(\mathbf T))\Vert\leq \Vert f\Vert_{\infty,\ov{\bl\Theta}_n}
=\Vert f\circ \bl\theta \Vert_{\infty,\ov{\mb D}^n}=\Vert q\Vert_{\infty,\ov{\mb D}^n}.
\eeqa
\end{proof}

We note down the following corollary which will be relevant in the sequel.

\begin{cor}\label{res}
Let $\mathbf T=(T_1,\ldots,T_n)$ be a commuting tuple of contractions on a Hilbert space $\m H$ satisfying  von Neumann's inequality for all $G$-invariant polynomials in $\mb C[z_1,\ldots,z_n].$ Let $\m M$ be a common invariant
subspace for $\theta_i(\mathbf T), i=1,\ldots,n.$ Then $(\theta_1(\mathbf T)\vert_{\m M},\ldots,\theta_n(\mathbf T)\vert_{\m M})$ is a
$\bl\Theta_n$-contraction on the Hilbert space $\m M.$
\end{cor}
\begin{proof}
 For a polynomial $f\in\mb C[z_1,\ldots,z_n],$ we note that
\beqa
\Vert f(\theta_1(\mathbf T)\vert_{\m M},\ldots, \theta_n(\mathbf T)\vert_{\m M})\Vert
&=&\Vert f(\theta_1(\mathbf T),\ldots, \theta_n(\mathbf T))\vert_{\m M}\Vert\\
&\leq& \Vert f(\theta_1(\mathbf T),\ldots, \theta_n(\mathbf T))\Vert\\
&\leq& \Vert f\Vert_{\infty,\bl\Theta_n}.
\eeqa
 Since $\bl\theta(\mathbf T)=(\theta_1(\mathbf T),\ldots,\theta_n(\mathbf T))$ is a $\bl\Theta_n$-contraction on $\m H$ by Proposition \ref{vN},  the last inequality holds. Hence we have the desired conclusion.
 \end{proof}
\begin{rem}
    We make a few immediate observations for a $\bl\Theta_n$-contraction below.
\begin{enumerate}\label{ex} \item [1.] Evidently, if $(T_1,\ldots,T_n)$ is a $\bl\Theta_n$-contraction on a
Hilbert space $\m H$ and $\m M$ is a common invariant subspace  for $T_i,i=1,\ldots,n,$ then
$(T_1\vert_{\m M},\ldots,T_n\vert_{\m M})$ is a $\bl\Theta_n$-contraction on $\m M.$

\item[2.]As an immediate consequence of Proposition \ref{vN}, we conclude that if $\mathbf T=(T_1,\ldots,T_n)$ is a commuting tuple in any of the classes of commuting contractions discussed in \cite{GKVW}, then $\bl\theta(\mathbf T)=(\theta_1(\mathbf T),\ldots,\theta_n(\mathbf T))$ is a  $\bl\Theta_n$-contraction. 

\item[3.] In view of \cite[Theorem 1.1]{Hartz}, if $\mathbf T-(T_1,\ldots, T_n)$ is an $n$-variable weighted shift such that each $T_j,\,\, j=1,\ldots, n$ is a contraction, then $\bl\theta(\mathbf T)=(\theta_1(\mathbf T),\ldots,\theta_n(\mathbf T))$ is a  $\bl\Theta_n$-contraction by Proposition \ref{vN}.
\end{enumerate}
\end{rem}


 $$
 $$
\section{$\bl\Theta_n$-unitaries and $\bl\Theta_n$-isometries}
We start with the following obvious generalizations of definitions in \cite{AY}.
\begin{defn}
 Let $T_1,\ldots,T_n$ be  commuting operators  on a Hilbert space $\mathcal H$. We say that $(T_1,\ldots,T_n)$ is
\begin{enumerate}
 \item [(i)] a $\bl\Theta_n$-{\it unitary} if $T_i, i=1,\ldots,n$ are normal operators and the joint spectrum
$\sigma(T_1,\ldots,T_n)$ of $(T_1,\ldots,T_n)$ is contained in the distinguished boundary of $\bl\Theta_n.$
\item[(ii)] a $\bl\Theta_n$-{\it isometry} if there exists a Hilbert space $\mathcal K$ containing $\mathcal H$ and a $\bl\Theta_n$-unitary $(\wi T_1,\ldots,\wi T_n)$ on $\m K$ such that $\m H$ is invariant for $\wi T_i, i=1,\ldots,n$
and $T_i={\wi T_i}\vert_{\m H}, i=1,\ldots,n.$
\item[(iii)] a $\bl\Theta_n$-{\it co-isometry} if $(T_1^*,\ldots,T_n^*)$ is a $\bl\Theta_n$-isometry.
\item[(iv)] a \emph{pure $\bl\Theta_n$-isometry} if $(T_1,\ldots,T_n)$ is a $\bl\Theta_n$-isometry and $T_n$ is a pure isometry.
\end{enumerate}

\end{defn}

The proof of the following theorem works along the lines of Agler and Young\cite{AY}.

 \begin{thm}\label{u}
  Let $T_j, \, j=1,\ldots,n,$ be commuting operators on a Hilbert space $\mathcal H.$ The following are equivalent:
  \begin{enumerate}
   \item[(i)] $(T_1,\ldots,T_n)$ is a $\bl\Theta_n$-unitary;
   \item[(ii)] $T_n^*T_n=I=T_nT_n^*,\, (T_n^p)^*T_j=T_{n-j}^*$ and $(\g_1T_1,\ldots,\g_{n-1}T_{n-1})$ is a $\Gamma_{n-1}$-contraction,
   where $\g_j=\frac{n-j}{n}$ for $j=1,\ldots,n-1$;
   \item[(iii)] there exist commuting unitary operators $U_j$ for $j=1,\ldots,n$ on $\mathcal H$ such that
   $$T_j=\sum_{1\leq k_1<\ldots<k_j\leq n}U_{k_1}^m\ldots U_{k_j}^m\mbox{~for~} j=1,\ldots,n-1 \mbox{~and~}  T_n^p=\big(\prod_{i=1}^n U_i\big)^m. $$
  \end{enumerate}
 \end{thm}

\begin{proof}Suppose (i) holds. Let $(T_1,\ldots, T_n)$ be a $\bl\Theta_n$-unitary. By the spectral theorem for commuting normal operators, there exists a spectral measure $E(\cdot)$ on $\sigma(T_1,\ldots,T_n)$ such that
 $$
 T_j = \int_{\sigma(T_1,\ldots,T_n)} z_j E(d\bl z),\, j=1,\ldots,n,
 $$
 where $z_1,...,z_n$ are the co-ordinate functions on $\C^n$. Let $\tau$ be a measurable right inverse of the restriction of $\bl \theta$ to $\mathbb T^n,$ so that $\tau$ maps distinguished boundary of $\bl\Theta_n$ to $\mathbb T^n$. Let $\tau = (\tau_1,\ldots,\tau_n)$ and
 $$
 U_j = \int_{\sigma(T_1,\ldots,T_n)}\tau_i(\bl z) E(d\bl z),\, j=1,\ldots,n.
 $$
 Clearly, $U_1,\ldots,U_n$ are commuting unitaries on $\mathcal H$ and
 \beqa
 \sum_{1\leq k_1<\ldots<kj\leq n}U_{k_1}^m\ldots U_{k_j}^m &=& \sum_{1\leq k_1<\ldots<k_j\leq n}\int_{\sigma(T_1,\ldots,T_n)}\tau_{k_1}(\bl z)\ldots \tau_{k_j}(\bl z) E(dz)\\ &=& \int_{\sigma(T_1,\ldots,T_n)}\theta_j(\bl z) E(d\bl z)\\ &=& T_j,
 \eeqa
 for $j=1,\ldots,n$. Hence (i) implies (iii).  

 Suppose (iii) holds. Then clearly,  $T_n^*T_n=I=T_nT_n^*$ and $ (T_n^p)^*T_j=T_{n-j}^*, j=1,\ldots,n.$ Since
 $(T_1,\ldots,T_n)$ is a $\bl\Theta_n$-contraction by Proposition \ref{vN}, it follows from Lemma \ref{projection} that $(\g_1T_1,\ldots,\g_{n-1}T_{n-1})$ is a $\Gamma_{n-1}$-contraction, where $\g_j=\frac{n-j}{n}$ for $j=1,\ldots,n-1.$ Hence (iii) implies (ii).

 Suppose (ii) holds. Since $T_n$ is normal, it follows from Fuglede's theorem that
 $$T_{n-i}^*T_{n-i} = (T_n^p)^*T_iT_i^*T_n^p = T_iT_i^*.$$ 
 Since $T_i$'s are  commuting operators, 
 $$T_{n-i}^*T_{n-i} = (T_n^p)^*T_iT_{n-i} = (T_n^p)^*T_{n-i}T_i =T_i^*T_i.$$  
 Thus $T_i, i= 1,\ldots,n$ is normal for $i=1,\ldots,n.$ Therefore, the unital $C^*$-algebra $C^*(T_1,\ldots, T_n)$ generated by $T_1,\ldots, T_n$ is commutative and by the Gelfand-Naimark's theorem is $*$-isometrically isomorphic to $C(\sigma(T_1,\ldots, T_n)).$ Let ${\hat T}_i$ be the image of $T_i$ for $i=1,\ldots,n,$ under the Gelfand map. By definition, for an arbitrary point ${\bl z} = (\theta_1,\ldots, \theta_n)$ in $\sigma(T_1,\ldots, T_n)$, ${\hat T}_i({\bl z}) = \theta_i$ for $i=1,\ldots,n.$ By the properties of the Gelfand map and the hypothesis we have,
$$
 \overline{\hat T_n({\bl z})}\hat T_n({\bl z})= 1 = \hat T_n({\bl z}) \overline{\hat T_n({\bl z})}\mbox{~and~} \overline{\hat T_n({\bl z})^p}\hat T_i({\bl z}) = \overline{\hat T_{n -i}({\bl z})}\mbox{~for~} \bl z \in\sigma(T_1,\ldots, T_n).
 $$
 Thus we obtain $\vert \theta_n\vert=1, {\bar \theta_n}^p \theta_i=\bar \theta_{n-i}.$ Now $(\g_1T_1,\ldots,\g_{n-1}T_{n-1})$ is a $\Gamma_{n-1}$-contraction implies
 $$
 \Vert f(\g_1T_1,\ldots,\g_{n-1}T_{n-1})\Vert \leq \Vert f\Vert_{\infty,\Gamma_{n-1}} \mbox{~for~} f\in\mb C[z_1,\ldots,z_n],
 $$
 which is equivalent to the fact that $\Vert f\Vert_{\infty,\Gamma_{n-1}}^2 - f(\g_1T_1,\ldots,\g_{n-1}T_{n-1})^*f(\g_1T_1,\ldots,\g_{n-1}T_{n-1})$ is a positive operator for $f\in\mb C[z_1,\ldots,z_n].$ An application of the Gelfand transform yields
 $$
 \Vert f\Vert_{\infty,\Gamma_{n-1}}^2 -\ov{ f(\g_1\hat T_1({\bl z}),\ldots,\g_{n-1}\hat T_{n-1}({\bl z}))}f(\g_1\hat T_1({\bl z}),\ldots,\g_{n-1}\hat T_{n-1}({\bl z}))\geq 0 \mbox{~for~} f\in\mb C[z_1,\ldots,z_n] $$
 and  $\bl z \in \sigma(T_1,\ldots, T_n).$ This proves that $(\g_1\theta_1,\ldots,\g_{n-1}\theta_{n-1})$ is in the polynomially convex hull of $\Gamma_{n-1}.$ Since $\Gamma_{n-1}$ is polynomially convex, $(\g_1\theta_1,\ldots,\g_{n-1}\theta_{n-1})$ is in $\Gamma_{n-1}.$ Therefore, by Theorem \ref{bg}, $(\theta_1,\ldots,\theta_n)$ is in the Shilov boundary of $\bl\Theta_n$ and hence $\sigma(T_1,\ldots, T_n)\subset \partial\bl\Theta_n$. This proves (ii) implies (i).
 \end{proof}

To proceed further, we introduce some spaces of vector-valued and operator-valued functions for our exposition, we recall them following \cite{AY}.
Let $\m E$ be a separable Hilbert space and  $\m L(\m E)$ be the $C^*$-algebra of bounded operators on $\m E,$ with the operator norm. Let $H^2(\m E) $ be the  Hardy space of analytic $\m E$-valued functions on $\mb D$ and $L^2(\m E)$ be the Hilbert space of square integrable $\m E$-valued functions on $\mb T,$ with their usual inner products. Let $H^\infty\m L(\m E)$ be the space of $\m L(\m E)$-valued bounded analytic  functions on $\mb D $ and $L^\infty\m L(\m E)$ the space of bounded measurable
$\m L(\m E)$-valued functions on $\mb T,$ each with appropriate version of the supremum norm. For
$\v\in L^\infty \m L(\m E)$ let  $T_\v$ be the Toeplitz operator with symbol $\v,$ given by
$$T_\v f=P(\v f),~~~~~ f\in H^2(\m E),$$
 $P:L^2(\m E)\lo H^2(\m E)$ being the orthogonal projection. In particular $T_z$ is the unilateral shift operator on
$H^2(\m E)$ ($z$ denotes the identity function on $\mb T$ ).

  We need the following lemma to characterize a a pure $\bl\Theta_n$-isometry.

\begin{lem}\label{taut}
 Let $\Phi_1,\ldots,\Phi_n$ be functions in $L^\infty\mathcal L(\mathcal E)$ and $M_{\Phi_i},\,i=1\ldots,n$, denote the corresponding multiplication operator on $L^2(\mathcal E).$ Then $(M_{\Phi_1},\ldots,M_{\Phi_n})$ is a $\bl\Theta_n$-contraction if and only if $(\Phi_1(z),\ldots,\Phi_n(z))$ is a $\bl\Theta_n$-contraction for all $z$ in $\mathbb T.$
\end{lem}
\begin{proof} Note that  $\|M_{\Psi}\| = \|\Psi\|_{\infty}$ for $\Psi\in L^\infty\mathcal L(\mathcal E)$ and  
$ f(M_{\Phi_1},\ldots,M_{\Phi_n})=M_{f(\Phi_1,\ldots,\Phi_n)}$ for all polynomials $f\in \mb C[z_1,\ldots,z_n].$ So
\Bea
 \Vert f(M_{\Phi_1},\ldots,M_{\Phi_n})\Vert=\Vert M_{f(\Phi_1,\ldots,\Phi_n)}\Vert=\Vert f(\Phi_1,\ldots,\Phi_n)\Vert.
\Eea
Thus 
\Bea
\|f(M_{\Phi_1},\ldots,M_{\Phi_n})\| \leq \Vert f\Vert_{\infty,\bl\Theta_n,}
\Eea
if and only of 
$$
\|f(\Phi_1(z),\ldots,\Phi_n(z))\| \leq \Vert f\Vert_{\infty,\bl\Theta_n,} \text{~for all~} z\in\mb T,
$$ which completes the proof.

\end{proof}

\begin{rem}
 An interesting case in the above Lemma occurs  when $\dim \mathcal E = 1.$ In this case, $(\Phi_1(z),\ldots,\Phi_n(z))$ is a $\bl\Theta_n$-contraction means that $(\Phi_1(z),\ldots,\Phi_n(z))\in\bl\Theta_n$ for all $z\in\mb T$ which is true as $\bl\Theta_n$ is polynomially convex.
\end{rem}

\begin{thm}\label{pure iso}
  Let $T_i, \, i=1,\ldots,n,$ be commuting operators on a Hilbert space $\mathcal H$. Then $(T_1,\ldots, T_n)$ is a pure $\bl\Theta_n$-isometry if and only if there exist a separable Hilbert space $\mathcal E$ and a unitary operator $U: \mathcal H\ra H^2(\mathcal E)$ and functions $\Phi_1,\ldots,\Phi_{n-1}$ in $H^\infty \mathcal L(\mathcal E)$ and operators $A_i \in \mathcal L(\mathcal E)\,i=1,\ldots,n-1$ such that
  \begin{enumerate}
   \item[(i)] $T_i = U^*T_{\Phi_i}U,\, i=1,\ldots,n-1,\, T_n = U^*T_zU$;
   \item[(ii)] $(\g_1\Phi_1(z),\ldots,\g_{n-1}\Phi_{n-1}(z))$ is a $\Gamma_{n-1}$-contraction for all $z$ in $\mathbb T,$ where $\g_i=\frac{n-i}{n}$ for $i=1,\ldots,n-1$;
   \item[(iii)] $\Phi_i(z) = \sum_{\ell=0}^pA_\ell^{(i)}z^\ell$ 
   for some $A_\ell^{(i)}\in\mathcal L(\mathcal E)$ satisfying ${A_\ell^{(i)}}^* = A^{(n-i)}_{p-\ell}$  for $\ell=0,\ldots,p$ and $i=1\ldots  n-1.$
 \item[(iv)] $\sum_{\ell=0}^k[A_\ell^{(i)}, A_{k-\ell}^{(j)}] = 0$  for $k=0,\ldots, 2p$ and $i,j=1\ldots   n-1,$ where $[P,Q]=PQ-QP$ for two operators $P,Q.$
  \end{enumerate}

 \end{thm}

  \begin{proof}
 Suppose $(T_1,\ldots, T_n)$ is a pure $\bl\Theta_n$-isometry. So, there exist a Hilbert space $\mathcal K$ and a $\bl\Theta_n$-unitary $(\tilde T_1,\ldots, \tilde T_n)$ such that $\mathcal H\subset\mathcal K$ is a common invariant subspace of $\tilde T_i$'s and $T_i = \tilde T_i|_{\mathcal H},\, i=1,\ldots,n$. From Theorem \ref{u}, it follows that $${\tilde T}_n^*{\tilde T}_n=I \mbox{~and~}
 {\tilde T}_i^*{\tilde T_n}^p={\tilde T}_{n-i}, \, 1\leq i\leq n-1.$$
 By compression of $\tilde T_i$ to $\mathcal H$, we have
 $$
 T_n^*T_n=I \mbox{~and~} T_i^*T_n^p=T_{n-i}, \, 1\leq i\leq n-1.
 $$
 Since $T_n$ is a pure isometry and $\mathcal H$ is separable, there exists a unitary operator $U: \mathcal H\ra H^2(\mathcal E)$ for some separable Hilbert space $\mathcal E$, such that $T_n = U^*T_zU$, where $T_z$ is the shift operator on $H^2(\mathcal E).$ Since $T_i$'s commute with $T_n$, there exist $\Phi_i\in H^\infty \mathcal L(\mathcal E)$ such that $T_i = U^*T_{\Phi_i}U$ for $1\leq i\leq n-1.$ Since  $(T_1,\ldots, T_n)$ is a $\bl\Theta_n$-contraction, an appeal to  Lemma \ref{projection} and Lemma \ref{taut} yields part (ii). 
 The relations $T_i^*T_n^p=T_{n-i}$ imply that
 $$
 T_{\Phi_i}^* T_z^p = T_{\Phi_{n-i}},\, 1\leq i\leq n-1.$$
 The last relations in turn yield
 $$\Phi_i(z)^*z^p=\Phi_{n-i}(z) \mbox{~for all~} z\in\mb T, \, i=1,\ldots, n-1.$$
 If $\Phi_i(z) = \sum_{\ell\geq 0}A_\ell^{(i)}z^\ell$ are in $H^\infty \mathcal L(\mathcal E)$ for $i=1,\ldots, n-1.$ Then $ {\Phi_i(z)}^*z^p=\Phi_{n-i}(z)$ for all $z\in\mb T$ implies

 $$
 \sum_{\ell\geq 0}{A_\ell^{(i)}}^* z^{p-\ell} = \sum_{\ell \geq 0}{A_\ell^{(n-i)}} z^\ell 
 \mbox{~for all~} z\in\mathbb T.$$ By comparison of the coefficients, we obtain
 $$
 A^{(i)}_\ell=A^{(n-i)}_\ell=0 \mbox{~for~}\ell>p \mbox{~and~} {A^{(i)}_\ell}^*=A^{(n-i)}_{p-\ell} \mbox{~for~} \ell=0,\ldots p,\,
 i=1,\ldots, n-1.$$   
 
 This shows that 
 $$\Phi_i(z)=\sum_{\ell=0}^p A_\ell^{(i)} z^\ell \mbox{~for some~} A_\ell^{(i)}\in\mathcal L(\mathcal E) \mbox{~satisfying~} {A_\ell^{(i)}}^* = A^{(n-i)}_{p-\ell}$$  for $\ell=0,\ldots,p$ and $i=1\ldots  n-1.$
 This proves part (iii). 
 
 Now since $T_i$'s commutes, we have  $T_{\Phi_i}T_{\Phi_j} = T_{\Phi_j}T_{\Phi_i}$ and consequently,
$$
\sum_{k=0}^{2p}\big(\sum_{\ell=0}^{k}A_\ell^{(i)}A_{k-\ell}^{(j)}\big)z^k=\sum_{k=0}^{2p}\big(\sum_{\ell=0}^{k}A_\ell^{(j)}A_{k-\ell}^{(i)}\big)z^k \mbox{~for all~} z\in\mathbb T.
$$
 Comparing  the coefficients of $z^k$ for $k=0,\ldots, 2p$ we get part (iv).

Conversely, let $(T_1,\ldots,T_n)$ be  a tuple satisfying conditions (i) to (iv). Consider the tuple $(T_{\Phi_1},\ldots,T_{\Phi_{n-1}},T_z)$ of multiplication operators on $L^2(\mathcal E)$ with symbols $\Phi_1,\ldots,\Phi_{n-1},z,$ respectively. The condition (iv) implies that $T_{\Phi_i}$'s are commuting operators. A part of condition (iii) shows that $(T_z^p)^* T_{\Phi_i} = T_{\Phi_{n-i}}^*$ by repeating analogous calculations as above. Thus, it follows from part (ii) of Theorem \ref{u}, that $(T_{\Phi_1},\ldots,T_{\Phi_{n-1}},T_z)$ is a $\bl\Theta_n$-unitary and so is $(U^*T_{\Phi_1}U,\ldots,U^*T_{\Phi_{n-1}}U,U^*T_zU)$. Now, $T_i$'s, $1\leq i\leq n$, are the restrictions to the common invariant subspace $H^2(\mathcal E)$ of $(T_{\Phi_1},\ldots,T_{\Phi_{n-1}},T_z)$ and hence $(T_1,\ldots,T_n)$ is a $\bl\Theta_n$-isometry. Since $T_n$ is a shift, $(T_1,\ldots,T_n)$ is a pure $\bl\Theta_n$-isometry.
 \end{proof}

 \begin{ex}
 For $n=2$ and $G(p,p, 2),\, \Phi_1(z)^*z^p=\Phi_1(z)$ for all $z \in \mb T.$ Putting $\Phi_1=\Phi$ and $\Phi(z)=\sum_{\ell= 0}^pA_\ell z^\ell.$ We have from Theorem \ref{pure iso} that $\Phi(z)=\sum_{\ell=0}^pA_\ell z^\ell$ and $A_\ell^*=A_{p-\ell}$ for $\ell=0,\ldots, p.$ In particular, for $p=2,\, \Phi(z)=A_0+A_1z+A_0^*z^2,$ where $A_0,A_1\in\m L(\m E)$ and $A_1^*=A_1.$
 \end{ex}

We now characterization of pure $\bl\Theta_n$-isometries in terms of the parameters
associated with them. For simplicity of notation, let $(T_\mathbf \Phi,T_z)$ denote the $n$-tuple of multiplication operators $(T_{\Phi_1},\ldots,T_{\Phi_{n-1}},T_z)$ on $H^2(\m E),$ where $\Phi_i\in H^\infty\m L(\m E),i=1,\ldots,n-1.$ We rely on the canonical identification of $H^2(\mathcal E)$ with $H^2\otimes\mathcal E$ by the map $z^n\xi \mapsto z^n\otimes\xi$, where $n\in\mathbb N\cup\{0\}$ and $\xi\in\mathcal E$, whenever necessary.

\begin{thm}\label{equi}
Let $\Phi_i(z) = \sum_{\ell=0}^pA_\ell^{(i)}z^\ell$ and $\tilde\Phi_i(z) = \sum_{\ell=0}^p\tilde A_\ell^{(i)}z^\ell$ be in  $H^\infty \mathcal L(\mathcal E)$ for some  $A_\ell^{(i)}\in\mathcal L(\mathcal E)$ and $\tilde A_\ell^{(i)}\in\mathcal L(\mathcal F)$ respectively, $i=1,\ldots,n-1$. Then the $n$-tuple $(T_\mathbf \Phi,T_z)$ on $H^2(\mathcal E)$ is unitarily equivalent to the $n$-tuple $(T_{\tilde{\mathbf\Phi}},T_z)$ on $H^2(\mathcal F)$ if and only if the tuples of operators $(A_\ell^{(i)})_{i=1,\ell = 0}^{{n-1}, p}$ and $(\tilde A_\ell^{(i)})_{i=1,\ell = 0}^{{n-1}, p}$ are unitarily equivalent.
\end{thm}

\begin{proof}
Suppose the $n$-tuple $(T_\mathbf \Phi,T_z)$ on $H^2(\mathcal E)$ is unitarily equivalent to the $n$-tuple $(T_{\tilde{\mathbf\Phi}},T_z)$ on $H^2(\mathcal F)$. We can identify the map $\Phi_i$ (similarly $\tilde\Phi_i$) by $I_{H^2}\otimes A_0^{(i)} +  \sum_{\ell=1}^p  T_{z^\ell}\otimes A_\ell^{(i)}$. So, there exist a unitary $U : H^2\otimes\mathcal E\ra H^2\otimes\mathcal F$ such that
\beq\label{ue}
U(I_{H^2}\otimes A_0^{(i)} +  \sum_{\ell=1}^p  T_{z^\ell}\otimes A_\ell^{(i)})U^* = I_{H^2}\otimes \tilde A_0^{(i)} +  \sum_{\ell=1}^p  T_{z^\ell}\otimes \tilde A_\ell^{(i)},\, i=1,\ldots,n-1,
\eeq and
\beq\label{ue1}
U(T_z\otimes I_{\mathcal E})U^* = T_z\otimes I_{\mathcal F}.
\eeq
 From equation \eqref{ue1}, it follows that there exists a unitary $\tilde U :\mathcal E\ra\mathcal F$ such that $U = I_{H^2}\otimes \tilde U$. Consequently, the equation \eqref{ue} can be written as
 $$
 \tilde U A_0^{(i)} \tilde U^* + \sum_{\ell=1}^p \tilde U A_\ell^{(i)}\tilde U^* z^\ell =  \tilde A_0^{(i)} + \sum_{\ell=1}^p \tilde A_\ell^{(i)} z^\ell,\,i=1,\ldots,n-1,
 $$ for all $z\in\mathbb T$. Hence comparing the coefficients, we obtain $\tilde U A_\ell^{(i)} \tilde U^* = \tilde A_\ell^{(i)},\,i=1,\ldots,n-1,\, \ell = 0,\ldots, p$, which completes the proof in forward direction.

Conversely, suppose there exist a unitary $\tilde U :\mathcal E\ra\mathcal F$ that intertwines $A_\ell^{(i)}$ and $\tilde A_\ell^{(i)}$, that is, $\tilde U A_\ell^{(i)} \tilde U^* = \tilde A_\ell^{(i)}$ for each $i = 1,\ldots,n-1,\, \ell = 0,\ldots, p$. Let $U: H^2\otimes\mathcal E\ra H^2\otimes\mathcal F$ be the map defined by $U = I_{H^2}\otimes \tilde U$. Clearly, $U$ is a unitary and from the computations similar to above, it is easy to see that $U$ intertwines $T_{\Phi_i}$ with $T_{\tilde \Phi_i}$ for each $i=,1\ldots,n-1$, and $T_z$. This completes the proof.
\end{proof}

\begin{rem}
From condition (iii) of Theorem \ref{pure iso}, it is enough to check the equivalence of tuples of operators for $i= 1, \ldots, [\frac{n}{2}]$.
\end{rem}

 In the next theorem, we characterize $\bl\Theta_n$-isometries by proving    an analogue of Wold decomposition for $\bl\Theta_n$-isometries.  We apply the following Lemma \cite[Lemma 2.5]{AY} in  proving our result.
 \begin{lem}\label{ay}
  Let $U$ and $V$ be a unitary and a pure isometry on Hilbert space $\mathcal H_1, \mathcal H_2$ respectively, and let $T:\mathcal H_1\ra\mathcal H_2$ be an operator such that $TU = VT$. Then $T = 0.$
 \end{lem}

 \begin{thm}\label{wold}
  Let $T_i, \, i=1,\ldots,n,$ be commuting operators on a Hilbert space $\mathcal H$. The following are equivalent:
  \begin{enumerate}
   \item [(i)] $(T_1,\ldots, T_n)$ is a $\bl\Theta_n$-isometry;
  \item [(ii)] $T_n^*T_n=I,\, (T_n^p)^*T_i=T_{n-i}^*$ and $(\g_1T_1,\ldots,\g_{n-1}T_{n-1})$ is a $\Gamma_{n-1}$-contraction,
   where $\g_i=\frac{n-i}{n}$ for $i=1,\ldots,n-1;$
  \item [(iii)] (Wold decomposition) there exists an orthogonal decomposition $\mathcal H = \mathcal H_1\oplus\mathcal H_2$ into common reducing subspaces of $T_i$'s, $i=1,\ldots,n$ such that $(T_1|_{{\mathcal H}_1},\ldots, T_n|_{{\mathcal H}_1})$ is a $\bl\Theta_n$-unitary and $(T_1|_{{\mathcal H}_2},\ldots, T_n|_{{\mathcal H}_2})$ is a pure $\bl\Theta_n$-isometry.
  \end{enumerate}
\end{thm}

\begin{proof} Suppose (i) holds. By definition, there exists $(\tilde T_1,\ldots,\tilde T_n)$ a $\bl\Theta_n$-unitary on $\mathcal K$ containing $\mathcal H$ such that $\mathcal H$ is a invariant subspace of $\tilde T_i$'s and $T_i$'s are restrictions of $\tilde T_i$'s to $\mathcal H.$ From Theorem \ref{u}, it follows that $\tilde T_i$'s are satisfying the relations:
$
{\tilde T}_n^*{\tilde T}_n=I,\, {\tilde T}_n^*{\tilde T}_i={\tilde T}_{n-i}^*
$
 and $(\g_1{\tilde T}_1,\ldots,\g_{n-1}{\tilde T}_{n-1})$ is a $\Gamma_{n-1}$-contraction,
   where $\g_i=\frac{n-i}{n}$ for $i=1,\ldots,n-1.$ Compressing to the common invariant subspace $\mathcal H$ and by part (1) of remark \ref{ex}, we obtain
 $$
 T_n^*T_n=I,\, T_n^*T_i=T_{n-i}^*
 $$
 and $(\g_1T_1,\ldots,\g_{n-1}T_{n-1})$ is a $\Gamma_{n-1}$-contraction, where $\g_i=\frac{n-i}{n}$ for $i=1,\ldots,n-1$. Thus (i) implies (ii).

 Suppose (ii) holds. By  Wold decomposition of $T_n,$ we  write $T_n = U\oplus V$ on $\mathcal H = \mathcal H_1\oplus\mathcal H_2,$ where $\mathcal H_1,\mathcal H_2$ are reducing subspaces for $T_n$, $U$ is unitary and $V$ is pure isometry. Let 
 $$
 T_i = \begin{bmatrix} T_{11}^{(i)} & T_{12}^{(i)}\\ T_{21}^{(i)} & T_{22}^{(i)} \end{bmatrix}
 $$
 with respect to this decomposition, where $T_{jk}^{(i)}$ is a bounded operator from $\mathcal H_k$ to $\mathcal H_j$. Since $T_nT_i = T_iT_n,$ it follows that
 $$
 \begin{bmatrix} UT_{11}^{(i)} & UT_{12}^{(i)}\\ VT_{21}^{(i)} & VT_{22}^{(i)} \end{bmatrix} = \begin{bmatrix} T_{11}^{(i)}U & T_{12}^{(i)}V\\ T_{21}^{(i)}U & T_{22}^{(i)}V \end{bmatrix},\, i = 1,\ldots,n-1.
 $$
 Thus, $T_{21}^{(i)}U = VT_{21}^{(i)}$ and hence by Lemma \ref{ay}, $T_{21}^{(i)} = 0,\, i = 1,\ldots,n-1$. Now $(T_n^p)^*T_i=T_{n-i}^*$ implies that
 $$
  \begin{bmatrix} (U^p)^*T_{11}^{(i)} & (U^p)^*T_{12}^{(i)}\\ 0 & (V^p)^*T_{22}^{(i)} \end{bmatrix} = \begin{bmatrix} T_{11}^{(n-i)*} & 0\\ T_{12}^{(n-i)*} & T_{22}^{(n-i)*} \end{bmatrix},\, i = 1,\ldots,n-1.
 $$
 It follows that $T_{12}^{(i)} = 0,\, i = 1,\ldots,n-1$. So $\mathcal H_1, \mathcal H_2$ are common reducing subspace for $T_1,\ldots,T_n.$ From the matrix equation above, we have $(U^p)^*T_{11}^{(i)} = T_{11}^{(n-i)*},\, i = 1,\dots,n-1.$ Thus by part (i) of the remark \ref{ex} and part (ii) of Theorem \ref{u}, it follows that $(T_{11}^{(1)},\ldots,T_{11}^{(n-1)},U)$ is a $\bl\Theta_n$-unitary.

 We now require to show that $(T_{22}^{(1)},\ldots,T_{22}^{(n-1)},V)$ is a pure $\bl\Theta_n$-isometry. Since V is a pure isometry $\mathcal H$ is separable, we can identify it with the shift operator $T_z$ on the Hardy space of vector valued functions $H^2(\mathcal E)$ for some separable Hilbert space $\mathcal E$. Since $T_{22}^{(i)}$'s commute with $V,$ there exists $\Phi_i\in H^\infty \mathcal L(\mathcal E)$ such that $T_{22}^{(i)}$ can be identified with $T_{\Phi_i}$ for $1\leq i\leq n-1$. Since $(T_1,\ldots, T_n)$ is a $\bl\Theta_n$-contraction, from Lemma \ref{projection} and Lemma \ref{taut}, it follows that $(\g_1T_{22}^{(1)},\ldots,\g_{n-1}T_{22}^{(n-1)})$ is a $\Gamma_{n-1}$-contraction,
   where $\g_i=\frac{n-i}{n}$ for $i=1,\ldots,n-1.$ The relations $(V^p)^*T_{22}^{(i)} = T_{22}^{(n-i)*}$ yield
 $$
 (T_{ z}^*)^p T_{\Phi_i} = T_{\Phi_{n-i}}^*,\, 1\leq i\leq n-1.
 $$
 Same Calculations as in the first part of the proof of Theorem \ref{pure iso} yields
 $\Phi_i(z) = \sum_{\ell=0}^pA_\ell^{(i)}z^\ell$ 
   for some $A_\ell^{(i)}\in\mathcal L(\mathcal E)$ satisfying ${A_\ell^{(i)}}^* = A^{(n-i)}_{p-\ell}$  for $\ell=0,\ldots,p$ and $i=1\ldots  n-1.$
 Now since $T_{22}^{(i)}$'s commutes, it follows that
$$\sum_{\ell=0}^k[A_\ell^{(i)}, A_{k-\ell}^{(j)}] = 0 \text{~for~} k=0,\ldots, 2p \text{~and~} i,j=1\ldots   n-1.$$ 
An appeal to Theorem \ref{pure iso}, $(T_1|_{{\mathcal H}_2},\ldots, T_n|_{{\mathcal H}_2})$ is a pure $\bl\Theta_n$-isometry. Thus (ii) implies (iii).

It is easy to see that (iii) implies (i).  This completes the proof.

\end{proof}

  \begin{cor}
   Let $T_i, \, i=1,\ldots,n,$ be commuting operators on a Hilbert space $\mathcal H$. Then $(T_1,\ldots, T_n)$ is a $\bl\Theta_n$-co-isometry if and only if
   $$
   T_nT_n^*=I,\, T_n^pT_j^*=T_{n-j} \text{~for~} j=1,\ldots, n-1,
   $$ 
   and $(\g_1T_1^*,\ldots,\g_{n-1}T_{n-1}^*)$ is a $\Gamma_{n-1}$-contraction,
   where $\g_j=\frac{n-j}{n}$ for $j=1,\ldots,n-1$.
  \end{cor}

\subsection{Characterization of invariant subspaces for $\bl\Theta_n$-isometries}

In view of Wold type decomposition in Theorem \ref{wold}, in order to characterize the invariant subspaces of $\bl\Theta_n$-isometries it is enough to characterize invariant subspaces for the associated model space obtained in  Theorem \ref{pure iso}. The following theorem discusses this issue.

A closed subspace $\mathcal M\neq \{0\}$ of $H^2(\m E)$ is said to be $(T_\mathbf \Phi,T_z)$-\emph{invariant} if $\m M$ is invariant
under $T_{\Phi_i}$ and $T_z$ for all $i=1,\ldots,n-1.$

Let $\m M\neq \{0\}$ be a closed subspace of $H^2(\m E_*).$ It follows from Beurling-Lax-Halmos theorem that $\m M$ is invariant under $T_z$ if and only
if there exists a Hilbert space  $\m E$ and an inner function $\Theta\in H^\infty\m L(\m E,\m E_*)$ ($\Theta $ is an isometry
almost everywhere on $\mb T$) such that
$$\m M=M_\Theta H^2(\m E).$$

 \begin{thm}
  Let $\mathcal M\neq \{0\}$ be a closed subspace of $H^2({\mathcal E_*})$ and $\Phi_i, i=1,\ldots,n-1$ be in
$H^\infty{\mathcal L(\mathcal E_*)}$ such that $(T_{{\mathbf \Phi}},T_z)$ is a pure $\bl\Theta_n$-isometry
on $H^2(\mathcal E_*)$. Then $\mathcal M$ is a $(T_{{\mathbf \Phi}},T_z)$-invariant subspace if and only
if there exist $\Psi_i,\, i=1,\ldots,n-1$ in $H^\infty{\mathcal L(\mathcal E)}$ such that
$(T_{{\mathbf \Psi}},T_z)$ is a pure $\bl\Theta_n$-isometry on $H^2(\mathcal E)$ and
  $$
  {\Phi_i}{\Theta} = {\Theta}{\Psi_i}, \, i=1,\ldots,n-1,
  $$ where $\Theta \in H^\infty{\mathcal L(\mathcal E, \mathcal E_*)}$ is the Beurling-Lax-Halmos representation of $\mathcal M$.
 \end{thm}
\begin{proof}
 We will prove only the forward direction as the other part is easy to see. Let $\mathcal M\neq \{0\}$ is invariant under $(T_{{\mathbf \Phi}},T_z)$. Thus, in particular, $\mathcal M$ is invariant under $T_z$ and hence the Beurling -Lax-Halmos theorem  $\mathcal M=\Theta H^2(\mathcal E),$ $\Theta\in H^\infty{\mathcal L(\mathcal E, \mathcal E_*)}$ being an inner multiplier. We also have $T_{\Phi_i}\Theta H^2(\mathcal E)\subseteq \Theta H^2(\mathcal E)$ for each $i=1,\ldots,n-1$. Thus, there exist $\Psi_i$'s, $i=1,\ldots,n-1$ in $H^\infty{\mathcal L(\mathcal E)}$ such that ${T_{\Phi_i}}{M_\Theta} = {M_\Theta}{T_{\Psi_i}},$ that is, ${\Phi_i}{\Theta} = {\Theta}{\Psi_i},\,i=1,\ldots,n-1$. Since $\Phi_i$'s commute, we have 
 $${M_\Theta}{T_{\Psi_i}}{T_{\Psi_j}} = {T_{\Phi_i}}{M_\Theta}{T_{\Psi_j}} = {T_{\Phi_j}}{T_{\Phi_i}}{M_\Theta} = {M_\Theta}{T_{\Psi_j}}{T_{\Psi_i}}$$ 
 and hence $\Psi_i\Psi_j = \Psi_j\Psi_i,\,i,j=1,\ldots,n-1$. Furthermore, we have
$$
f(T_{\Phi_1},\ldots,T_{\Phi_{n-1}})M_{\Theta} =  M_{\Theta}f(T_{\Psi_1},\ldots,T_{\Psi_{n-1}})
$$ for all polynomials $f\in \mb C[z_1,\ldots,z_{n-1}].$ Therefore,
$$
\|f(\g_1T_{\Psi_1},\ldots,\g_{n-1}T_{\Psi_{n-1}})\|\leq \|M_\Theta^*f(\g_1T_{\Phi_1},\ldots,\g_{n-1}T_{\Phi_{n-1}})M_{\Theta}\| \leq \Vert f\Vert_{\infty,\Gamma_{n-1},}
$$ for all polynomials $f\in \mb C[z_1,\ldots,z_{n-1}]$ and $\g_i=\frac{n-i}{n}$ for $i=1,\ldots,n-1$. Using Theorem \ref{pure iso}, we also note that
$$
(T_ z^p)^* T_{\Psi_i} = (T_ z^p)^*M_{\Theta}^*T_{\Phi_i}M_{\Theta} = M_{\Theta}^*T_{\bar z}T_{\Phi_i}M_{\Theta} = M_{\Theta}^*T_{\Phi_{n-i}}^*M_{\Theta} = T_{\Psi_{n-i}}^*,\, i=1,\ldots,n-1.
$$ From the observations made above and by Theorem \ref{pure iso}, it follows that $(T_{{\mathbf \Psi}},T_z)$ is a pure $\bl\Theta_n$-isometry on $H^2(\mathcal E)$ and this completes the proof.
\end{proof}

\section{Functional models on quotient domains}

We start this section by quoting a few results from \cite{BDGS}. We recall two useful definitions following \cite{DP} to pursue the exposition in \cite{BDGS}.

 \begin{defn}
A Hilbert space $\mathcal H$ is said to be a {\rm Hilbert module} over an algebra $\m A$ if the map $(f,h) \to f\cdot h,$ $f\in \m A, h\in \mathcal H,$ defines an algebra homomorphism $f \mapsto T_f$ of $\m A$ into $\m L(\m H),$ where $T_f$ is the bounded operator defined by  $T_f h = f \cdot h.$
\end{defn}
A closed subspace $\m M$ of $\m H$ is said to be a {\rm submodule} of the Hilbert module $\m H$ if $\m M$ is closed under the action of $\m A.$ If in addition, $T_f^*\m M\subseteq\m M$ for all $f\in\m A,$ then $\m M$ is called a {\rm reducing submodule} of $\m H,$ where $T_f^*$ denotes the adjoint of $T_f.$ A reducing submodule $\m M$ is called \emph{minimal} or \emph{irreducible} if only reducing submodules of $\m M$ are  $\{0\}$ and $\m M$ itself. It is easy to see that $\m H$ admits a reducing  submodule if and only if there exists an orthogonal projection $P $ on $\m H$ such that $PT_f=T_fP$ for all $f\in\m A.$

\begin{defn}
A Hilbert module $\m H$ over $\mb C[z_1,\ldots,z_n]$ of $\mb C^m$-valued holomorphic functions on $\Omega\subseteq\mb C^n,$ is said to be an {\it analytic Hilbert module}  if
\begin{enumerate}
\item [(1)] $\mb C[z_1,\ldots,z_n]\otimes\mb C^m$  is dense in $\m H$ and
\item[(2)] $\m H$ possesses a $\m L(\mb C^m)$-valued reproducing kernel on $\Omega.$
\end{enumerate}
The module action in an analytic Hilbert module is given by pointwise multiplication, that is, ${\mathfrak m}_p(\bl h)(\bl z) = p(\bl z) {\bl h}(\bl z),\, \bl z\in \Omega.$ Note that if the $i$-th component of $\bl h\in \m H$ is $h_i,$ then the $i$-th component of ${\mathfrak m}_p(\bl h)(\bl z)$ is $p(\bl z)h_i(\bl z)$ for $1\leq i\leq m.$  
\end{defn}
Let $G$ be a finite group which acts on the open unit polydisc $\mb D^n.$ Suppose that $\m H$ is a Hilbert space consisting of holomorphic functions on $\mb D^n$ with a $G$-invariant reproducing kernel $K,$ that is, $K$ satisfies $$ K(\sigma\cdot z,\sigma\cdot w)=K(z, w) \text{~for all~} z,w\in \mb D^n \text{~and~} \sigma\in G.$$
Let $\widehat{G}$ be the set of equivalence class of irreducible representations of $G.$ For each $\varrho\in\widehat G,$ we define the linear map
$\mb P_\varrho:\m H\to\m H$ by 
\bea\label{proje}
(\mb P_\varrho f)(z)=\frac{\deg \varrho}{\vert G\vert}\sum_{\sigma\in G}\chi_\varrho(\sigma^{-1})f(\sigma^{-1}\cdot z),
\eea
where $\chi_\varrho$ is the character of the representation $\varrho$ and $\deg \varrho$ denotes the dimension of $\varrho.$

\begin{rem}
      The \emph{sign representation} of a finite complex reflection group $G,$ $\rm sign : G \to \mb C^*,$ is defined by \cite[p. 139, Remark (1)]{Stanley} \begin{eqnarray}\label{sign} \rm sign(\sigma) = (\det(\sigma))^{-1},\,\,\,\, \sigma \in G.\end{eqnarray} Any element $f \in \mb P_{\rm sign}(\m H)$ can be written as $f = J_{\bl \theta} ~ \widetilde{f} \circ \bl \theta$ for a holomorphic function $\widetilde{f}$ on $\bl \theta(\mb D^n),$ where $J_{\bl \theta}$ is the determinant of the complex Jacobian matrix of the basic polynomial map $\bl \theta$.
  \end{rem}

If $\deg\varrho>1,$ we can split $\mb P_\varrho$ further as follows. Fix  $\deg\varrho=m>1.$ Suppose $\pi_\varrho$ is a unitary representation of $G$ which belongs to the equivalence class $\varrho$ in $\widehat G.$ Then, for some choice of an orthonormal basis of $\mb C^m,$ the matrix representation of $\pi_\varrho$ is given by 
\Bea
\pi_\varrho(\sigma)=(\!\!(\pi_\varrho^{ij}(\sigma))\!\!)_{i,j=1}^m\in\mb C^{m\times m}\,\,\text{ for } \sigma \in G.
\Eea
The linear maps $\mb P_\varrho^{ij}:\m H\to\m H$ for $i,j=1,\ldots, m,$ are given by
\Bea
(\mb P_{\varrho}^{ij}f)(z)=\frac{\deg \varrho}{\vert G\vert}\sum_{\sigma\in G}\pi_{\varrho}^{ji}(\sigma^{-1})f(\sigma^{-1}\cdot z).
\Eea We note that $\chi_\varrho(\sigma)={\rm trace}(\pi_\varrho(\sigma))$ for every $\sigma \in G$ and hence evidently $\mb P_\varrho = \sum_{i=1}^m \mb P_\varrho^{ii}.$

If the group $G$ is a finite complex reflection group then a basic polynomial mapping $\bl\theta=(\theta_1,\ldots,\theta_n):\mb D^n \to\bl\Theta_n$ is a $G$-invariant proper holomorphic mapping \cite[Proposition 5.4]{BDGS}.


\begin{prop}\cite[Theorem 5.7]{BDGS}\label{reduce}
Let $G$ be a finite  complex reflection group which acts on $\mb D^n$ and  $\m H$ be an analytic Hilbert module on $\mb D^n$ (with a $G$-invariant reproducing kernel) over $\mb C[z_1,\ldots,z_n]$. Then each $\mb P_\varrho^{ii}\m H$ is a joint reducing submodule over the polynomial ring $\mb C[z_1,\ldots,z_n]^G$ for every $\varrho\in\widehat G$ and $1\leq i\leq \deg \varrho.$ 

Consequently, the number of joint reducing submodules of $\m H$ over the ring $\mb C[z_1,\ldots,z_n]^G$ is at least $\sum_{\varrho\in\widehat G}\deg \varrho.$
\end{prop}

Henceforth, we specialize to $G=G(m,p,n),$ here $m,p,n$ are positive integers, $p$ divides $m$ and $n>1.$ Then the basic polynomial mapping $\bl\theta:\mb D^n\to\bl\theta(\mb D^n)$ associated to $G$ is such that
\bea\label{basicpoly}
\ov{\theta_i(z)}\theta_n^p(z)=\theta_{n-i}(z) \text{~for~} z\in\mb T^n.
\eea
A combination of Proposition \ref{reduce} and Corollary \ref{res} allows us to conclude the next result which in turn enables us to exhibit a plethora of examples of $\bl\Theta_n$-contractions.
\begin{thm}\label{reduced}
 Suppose that $G(m,p,n)$ acts on $\mb D^n$ and $\m H$ is an analytic Hilbert module on $\mb D^n$ over $\mb C[z_1,\ldots,z_n]$ with a $G(m,p,n)$-invariant reproducing kernel. 
 \begin{enumerate}
     \item If the  multiplication operators $M_{z_1},\ldots, M_{z_n}$  by the coordinate functions on $\m H$ are contractions and satisfy the analogue of von Neumann's inequality for all $G(m,p,n)$-invariant polynomials in $\mb C[z_1,\ldots, z_n],$ then  the commuting $n$-tuple of the restriction operators $(M_{\theta_1},\ldots, M_{\theta_n})|_{\mb P_\varrho^{ii}\m H}$ is a $\bl\Theta_n$-contraction for every   $\varrho\in\widehat G$ and $1\leq i\leq \deg \varrho.$  
     \item Consequently, at least $\sum_{\varrho\in\widehat G}\deg \varrho$ number of $\bl\Theta_n$-contractions arises from $(M_{z_1},\ldots, M_{z_n})$ in this way.
 \end{enumerate}
\end{thm}

\subsection{Unitary equivalence} Theorem \ref{reduced} provides examples of classes of $\bl\Theta_n$-contractions. Proposition \ref{inequihilmod} and Theorem \ref{sm} concern about natural question of the unitary equivalence of those $\bl\Theta_n$-contractions. 

\begin{prop}\cite[Theorem 6.10]{BDGS}\label{inequihilmod}
Let $G$ and $\m H$ be an analytic Hilbert module on $\mb D^n$ with a  $G$-invariant kernel. If $\varrho$ and $\varrho^\i\in\widehat{G}$ are such that $ {\rm deg}\,{\varrho}\neq  {\rm deg}\,{\varrho^\i},$ then
\begin{enumerate}
\item[$(1)$] $\mb P_{\varrho}\mathcal H$ and $\mb P_{\varrho^\i}\mathcal H$ are not unitarily equivalent submodules over the ring of $G$-invariant polynomials, and
\item[$(2)$] the submodules  $\mb P_{\varrho}^{ii}\m H$ and $\mb P_{\varrho^\i}^{jj}\m H$ are not unitarily equivalent for any $i,j$ with $1\leq i \leq  {\rm deg}\,{\varrho}$ and $1\leq j \leq  {\rm deg}\,{\varrho^\i}.$
\item[$(3)$] However, the submodules $\mb P_{\varrho}^{ii}\m H$ and $\mb P_{\varrho}^{jj}\m H$ are unitarily equivalent for any $\varrho\in\widehat{G}$ and $i,j$ with $1\leq i,j \leq  {\rm deg}\,{\varrho}.$
\end{enumerate}
\end{prop}

Under the additional hypothesis that $\m H \subseteq L^2(\mb D^n,dm),$ where $L^2(\mb D^n,dm)$ denotes the space of all measurable, square integrable functions with respect to some Borel measure $dm$ on $\mb D^n$, the following result includes the possibility even when ${\rm deg}\,{\varrho} = {\rm deg}\,{\varrho^\i}$. We set the restriction operator ${M_{\theta_i}^{(\varrho,j)}}:=M_{\theta_i} \big|_{\mb P_{\varrho}^{jj} \mathcal H} .$ 
\begin{thm}\label{sm}
Let $\varrho, \varrho^\i$ be distinct elements in $\widehat{G}.$ Then the tuples of restriction operators $({M_{\theta_1}^{(\varrho,j)}},\ldots,{M_{\theta_n}^{(\varrho,j)}})$  and $({M_{\theta_1}^{(\varrho^\i,k)}},\ldots{M_{\theta_n}^{(\varrho^\i,k)}})$ are not unitarily equivalent for $j=1,\ldots,{\rm deg}\,\varrho$ and $k=1,\ldots,{\rm deg}\,\varrho^\i$.
\end{thm}
\begin{proof}
Arguing by contradiction, suppose not, then there exists a unitary $U : \mb P_\varrho \m H \to \mb P_{\varrho^\i} \m H$ which intertwines ${M_{\theta_i}^{(\varrho^\i,k)}}$ and ${M_{\theta_i}^{(\varrho,j)}}.$ Suppose $U(f) = g$ for some $f \in \mb P_\varrho \m H$ and $g \in \mb P_{\varrho^\i} \m H.$ Then $\norm{f} = \norm{g}$ implies that 
\Bea \int_{\mb D^n} |q(\bl z)|^2 |f(\bl z)|^2 dm(\bl z)= \int_{\mb D^n} |q(\bl z)|^2 |g(\bl z)|^2 dm(\bl z) 
\Eea 
for all $G$-invariant polynomial $q$ in $n$ variables. Note that $|J_\sigma(\bl z)|^2 = 1 \text{ for all } \sigma \in G \text{ and } \bl z \in \mb D^n.$ Therefore, for all $\sigma \in G,$  
\Bea \int_{\mb D^n} |q(\bl z)|^2 |f(\sigma \cdot \bl z)|^2 dm(\bl z) &=& \int_{\mb D^n} |q(\bl z)|^2 |g(\sigma \cdot \bl z)|^2 dm(\bl z) \\ \int_{\mb D^n} |q(\bl z)|^2 \sum_{\sigma \in G}|f(\sigma \cdot \bl z)|^2 dm(\bl z) &=& \int_{\mb D^n} |q(\bl z)|^2 \sum_{\sigma \in G}|g(\sigma \cdot \bl z)|^2 dm(\bl z), 
\Eea 
and thus for every $G$-invariant polynomial $q,$ we get 
\Bea \displaystyle \int_{\mb D^n} |q(\bl z)|^2 \sum_{\sigma \in G}(|f(\sigma \cdot \bl z)|^2 - |g(\sigma \cdot \bl z)|^2) dm(\bl z) = 0,
\Eea 
equivalently,  
\bea\label{nor} \sum_{\sigma \in G}|f(\sigma \cdot \bl z)|^2 = \sum_{\sigma \in G}|g(\sigma \cdot \bl z)|^2.
\eea

Let enlist the elements of the group $G$ by $\{\sigma_i:i=1,\ldots,d=|G|\}.$ We define two functions $F_i: \mb D^n \to \mb C^d$ for $i=1,2,$ by \Bea F_1(\bl z) &=& (f(\sigma_1 \cdot \bl z),\ldots, f(\sigma_d \cdot \bl z))\\ F_2(\bl z) &=& (g(\sigma_1 \cdot \bl z),\ldots, g(\sigma_d \cdot \bl z)).\Eea From Equation \eqref{nor}, $\norm{F_1(\bl z)}_{\mb C^d}^2 = \sum_{\sigma \in G}|f(\sigma \cdot \bl z)|^2 = \sum_{\sigma \in G}|g(\sigma \cdot \bl z)|^2 = \norm{F_2(\bl z)}^2_{\mb C^d}.$ Let the closed linear hull $\{F_i(\bl z) : \bl z \in \mb D^n\}$ is denoted by $V_i, i=1,2.$ Then \cite[Lemma 0.1, p. 3193]{Fricain} ensures the existence of an isometry $T$ from $V_1$ onto $V_2$ such that \bea\label{tf}TF_1(\bl z) = F_2(\bl z),~~ \text{ for all } \bl z \in \mb D^n.\eea We extend $T$ to be an invertible operator on the whole $\mb C^d$ by defining it identity to the orthogonal complement of $V_1.$ Since $T$ is invertible and $T$ follows Equation \eqref{tf}, we have $${\rm span}_{\mathbb C}\{f(\sigma \cdot \bl z) : \sigma \in G\}  =  {\rm span}_{\mathbb C}\{g(\sigma \cdot \bl z) : \sigma \in G\}.$$ This leads to a contradiction as ${\rm span}\{f(\sigma \cdot \bl z) : \sigma \in G\}$ decomposes to copies of $\varrho$ and ${\rm span}\{g(\sigma \cdot \bl z) : \sigma \in G\}$ decomposes to copies of $\varrho^\i.$
\end{proof}
\begin{rem} Let $\varrho,\varrho^\i \in \widehat{G}$ be two inequivalent representations. Theorem \ref{sm} establishes that the submodules  $\mb P_{\varrho}^{ii}\m H$ and $\mb P_{\varrho^\i}^{jj}\m H$ are not unitarily equivalent over $\mb C[\bl z]^G$ for any $i,j$ with $1\leq i \leq  {\rm deg}\,{\varrho}$ and $1\leq j \leq  {\rm deg}\,{\varrho^\i}.$ 
    \begin{enumerate}
        \item If ${\rm deg}\, \varrho = {\rm deg}\, \varrho^\i= 1,$ then it states that $\mb P_{\varrho}\mathcal H$ and $\mb P_{\varrho^\i}\mathcal H$ are not unitarily equivalent submodules over the ring of $G$-invariant polynomials. 
        \item Next it is natural to inspect whether the reducing submodules $\mathbb P_{\varrho}\mathcal H$ and $\mathbb P_{\varrho^\i}\mathcal H$ are similar. Corollary \ref{simlr} answers it in negative.
        \item In a similar way, one can prove the statement of Theorem \ref{sm} for any Hilbert space with afore-mentioned properties consisting holomorphic functions on $\Omega,$ where $\Omega$ is any domain in $\mb C^n.$
   \end{enumerate}
\end{rem}
The following corollary is an immediate consequence of Theorem \ref{reduced}, Proposition \ref{inequihilmod} and Theorem \ref{sm} which is interesting enough to be stated separately.
\begin{cor}\label{number}
Suppose that $\m H$ is an analytic Hilbert module such that $\m H\subseteq L^2(\mb D^n, dm)$ for some Borel measure on $\mb D^n.$ Then, there are precisely  $\vert\widehat G\vert$ number of mutually unitarily inequivalent families $\bl\Theta_n$-contractions each of which consists of $\deg\varrho$ number of (unitarily) equivalent $\bl\Theta_n$-contractions.  
\end{cor}




\subsection{The weighted Bergman module} For $\l> 1,$ let $dV^{(\lambda)}:= \big(\frac{\lambda-1}{\pi}\big )^n \prod_{i=1}^n(1-\vert z_i\vert^2)^{\lambda-2} dV$ be a measure on the polydisc $\mb D^n,\,\, dV$ being the Lebesgue measure on $\mb D^n.$  Let
$$\mathbb A^{(\lambda)}(\mathbb D^n):=\{f\in L^2(\mathbb  D^n,\,dV^{(\lambda)}): f \text{~is holomorphic~}  \}$$
be the weighted Bergman space on the polydisc $\mathbb D^n.$  The reproducing kernel $K_\l$ of $\mathbb A^{(\lambda)}(\mathbb D^n)$ is given by
$$K_\l(\bl z,\bl w)=\prod_{i=1}^n(1-z_i\bar w_i)^{-\l},\bl z=(z_1,\ldots,z_n),\bl w=(w_1,\ldots,w_n)\in \mb D^n.$$
The limiting case of $\lambda=1$ is the Hardy space $H^2(\mb D^n)$ on the polydisc $\mb D^n$ and the reproducing kernel of $H^2(\mb D^n)$ is given by $$K_1(\bl z,\bl w)=\prod_{i=1}^n(1-z_i\bar w_i)^{-1},\bl z=(z_1,\ldots,z_n),\bl w=(w_1,\ldots,w_n)\in \mb D^n.$$ 

 It is clear that $K_\l$ remains invariant under the action of every complex reflection in every $G(m,p,n)$ and hence by $G(m,p,n)$. Moreover, each $\mathbb A^{(\lambda)}(\mathbb D^n)$ is an analytic Hilbert module over the polynomial algebra for $\l\geq1$. In the sequel, we focus on certain $n$-tuple of operators on the weighted Bergman spaces $\mathbb A^{(\lambda)}(\mathbb D^n)$ for $\l\geq 1.$

\begin{rem}
    Let $(X,\mu)$ be a measure space with measure $\mu.$ For $\varphi\in L^\infty(X,\mu),$ let $M_{\varphi}$ denote the operator on $L^2(X,\mu)$ defined by $M_\varphi f= \varphi f$ for $f\in L^2(X,\mu).$ Then $\Vert M_\varphi\Vert\leq \Vert\varphi\Vert_\infty.$ We use this observation in our next proof.
\end{rem}
\begin{prop}\label{vonN}
 For $\l\geq 1,$ 
\begin{enumerate}
\item[(i)] $\mathbf M_{\bl\theta}=(M_{\theta_1},\ldots, M_{\theta_n})$ is a $\bl\Theta_n$-contraction on $L^2(\mb D^n,dV^{(\l)}),$
\item[(ii)] the restriction of the tuple $\mathbf M_{\bl\theta}$ to $\mb A^{(\l)}(\mb D^n)$ is a $\bl\Theta_n$-contraction.
\end{enumerate}
\end{prop}
\begin{proof}
Let $\mathbf M=(M_{z_1},\ldots, M_{z_n})$ be the tuple of multiplication operators by the coordinate functions on $L^2(\mb D^n,dV^{(\l)}).$
For $\l\geq 1$ and $f\in\mb C[z_1,\ldots, z_n],$ it follows that

\Bea
 \Vert f(\mathbf M)\Vert=\Vert M_{f}\Vert\leq\Vert f\Vert_{\infty,\ov{\mb D}^n}.
\Eea
Hence (i) follows immediately from Proposition \ref{vN}.

Since $\mb A^{(\l)}(\mb D^n)$ is a joint invariant subspace of $\mathbf M_{\bl\theta},$ (ii) follows from Corollary \ref{res}.
\end{proof}

The next result follows immediately from Theorem \ref{reduced}, Proposition \ref{vonN} and Corollary \ref{number}. This exhibits a large class of reproducing kernel Hilbert spaces on each of which the naturally occurring tuple of multiplication operators is a $\bl\Theta_n$-contraction.

\begin{thm}\label{Theta}
  The commuting tuple $(M_{\theta_1},\ldots,M_{\theta_n})$ of multiplication operators acting on $\mb P_{\varrho}^{ii}\big(\mb A^{(\l)}(\mb D^n)\big)$ is a $\bl\Theta_n$-contraction for every $\varrho\in\widehat G,\,\, \text{ for } G=G(m,p,n),\,\, 1\leq i\leq\deg\varrho $  and all $\l\geq 1.$ Moreover, for every $\l\geq 1,$  there are precisely  $\vert\widehat G\vert$ number of mutually unitarily inequivalent families $\bl\Theta_n$-contractions each of which consists of (unitarily) equivalent $\bl\Theta_n$-contractions.  
\end{thm}

The Shilov boundary of the polydisc $\mb D^n$ is the $n$-torus $\mb T^n.$ Let $L^2(\mb T^n)$ denote the vector space of square integrable functions with respect to the Lebesgue measure on $\mb T^n.$ The linear map $\mb P_\varrho:L^2(\mb T^n)\to L^2(\mb T^n)$ defined by Equation \eqref{proje}  satisfies $\mb P_\varrho^2=\mb P_\varrho^*=\mb P_\varrho$ for every $\varrho\in\widehat G$ \cite[p.10]{GS} and
\Bea
L^2(\mb T^n)=\oplus_{\varrho\in\widehat G}\mb P_\varrho(L^2(\mb T^n)).
\Eea


\begin{lem}\label{redu}
Suppose that $G$ acts on $\mb T^n,$ then $M_{\theta_k}\mb P_\varrho=\mb P_\varrho M_{\theta_k}$ on $L^2(\mb T^n)$ for $\varrho\in\widehat G$ and $k=1,\ldots, n.$ Equivalently, $\mb P_\varrho(L^2(\mb T^n)$ is a joint reducing subspace of $( M_{\theta_1},\ldots,M_{\theta_n})$ for every $\varrho\in\widehat G.$
\end{lem}
\begin{proof}
For $f\in L^2(\mb T^n)$ and $k=1,\ldots, n,$ it follows from Equation \eqref{proje} that 
\Bea
(\mb P_\varrho M_{\theta_k}f)(z)&=& \frac{\deg \varrho}{\vert G\vert}\sum_{\sigma\in\widehat G}\chi_\varrho(\sigma^{-1})\theta_k(\sigma^{-1}\cdot z)f(\sigma^{-1}\cdot z)\\
&=&  \frac{\deg \varrho}{\vert G\vert}\sum_{\sigma\in\widehat G}\chi_\varrho(\sigma^{-1})\theta_k(z)f(\sigma^{-1}\cdot z)\\
&=& \theta_k(z)\frac{\deg \varrho}{\vert G\vert}\sum_{\sigma\in\widehat G}\chi_\varrho(\sigma^{-1})f(\sigma^{-1}\cdot z)\\
&=& (M_{\theta_k}\mb P_\varrho)f)(z),
\Eea
where the second equality follows from $G$-invariance of $\bl\theta=(\theta_1,\ldots.\theta_n).$
\end{proof}
The following result exhibits examples of $\bl\Theta_n$-unitaries.
\begin{thm}
\begin{enumerate}
\item[(i)]The commuting tuple $(M_{\theta_1},\ldots,M_{\theta_n})$ of multiplication operators acting on  $L^2(\mb T^n)$ is a $\bl\Theta_n$-unitary,
\item[(ii)] the restriction of $(M_{\theta_1},\ldots,M_{\theta_n})$ to $\mb P_\varrho\big(L^2(\mb T^n)\big)$ is a $\bl\Theta_n$-unitary for every $\varrho\in\widehat G.$
\end{enumerate}
 
\end{thm}
\begin{proof}
Since the multiplication operator by the coordinate function $M_{z_i}$ is a unitary operator on $L^2(\mb T^n),$ for $i=1,\ldots, n,$  the commuting tuple $(M_{z_1},\ldots, M_{z_n}) $ satisfies the von Neumann inequality. Therefore, $(M_{\theta_1},\ldots,M_{\theta_n})$ on $L^2(\mb T^n)$ is a $\bl\Theta_n$-contraction by Proposition \ref{vN}. Moreover, Equation \eqref{basicpoly} implies that $(M_{\theta_1},\ldots,M_{\theta_n})$ satisfies 
\bea\label{gu}
M_{\theta_n}^*M_{\theta_n}=I=M_{\theta_n}^*M_{\theta_n} \text{~and~}  (M^p_{\theta_n})^*M_{\theta_j}=M_{\theta_{n-j}} \text{~for~} j=1,\ldots,n-1.
\eea
An appeal to Lemma \ref{projection} Theorem \ref{u} completes the proof of (i).

Lemma \ref{redu} implies that the restriction of $(M_{\theta_1},\ldots,M_{\theta_n})$ to $\mb P_\varrho\big(L^2(\mb T^n)\big)$ is a $\bl\Theta_n$-contraction satisfying Equation \eqref{gu}. Thus (ii) follows by another application  of Lemma \ref{projection} and Theorem \ref{u}.
\end{proof}

It is well-known that the Hardy space $H^2(\mb D^n)$ of the polydisc $\mb D^n$ can be regarded as a subspace of $L^2(\mb T^n).$ In the following result, we describe a family of $\bl\Theta_n$-isometries.

\begin{cor}\label{iso}
\begin{enumerate}
\item[(i)] The commuting tuple $(M_{\theta_1},\ldots,M_{\theta_n})$ of multiplication operators acting on $H^2(\mb D^n)$ is a $\bl\Theta_n$-isometry. 
\item [(ii)] the restriction of  $(M_{\theta_1},\ldots,M_{\theta_n})$ to $\mb P_{\varrho}^{ii}\big(H^2(\mb D^n)\big)$ is a $\bl\Theta_n$-isometry for every $\varrho\in\widehat G$ and $1\leq i\leq\deg\varrho.$  
\end{enumerate}
\end{cor}
\begin{proof}
The tuple $(M_{\theta_1},\ldots,M_{\theta_n})$ on $H^2(\mb D^n)$ is the restriction of the $\bl\Theta_n$-unitary on $L^2(\mb T^n).$ Thus (i) follows from the definition of  $\bl\Theta_n$-isometry.

It is immediate from Theorem \ref{reduced} and Theorem \ref{wold} that $(M_{\theta_1},\ldots,M_{\theta_n})\vert_{\mb P_\varrho^{ii}\big(H^2(\mathbb D^n)\big)}$ is a $\bl\Theta_n$-contraction satisfying 
\Bea
M_{\theta_n}^*M_{\theta_n}=I\text{~and~}  (M^p_{\theta_n})^*M_{\theta_j}=M_{\theta_{n-j}} \text{~for~} j=1,\ldots,n-1.
\Eea
Hence $(M_{\theta_1},\ldots,M_{\theta_n})\vert_{\mb P_\varrho^{ii}\big(H^2(\mathbb D^n)\big)}$ is a $\bl\Theta_n$-isometry by Lemma \ref{projection} and Theorem \ref{wold}.  This proves (ii).
\end{proof}
Combining Corollary \ref{number} and Corollary \ref{iso}, we conclude the following result.
\begin{cor}
The collection $\{(M_{\theta_1},\ldots,M_{\theta_n})\vert_{\mb P_\varrho^{ii}\big(H^2(\mb D^n)\big)}:\varrho\in\widehat G,\,\,1\leq i\leq\deg\varrho\}$ contains precisely $\vert \widehat G\vert$ number of mutually inequivalent $\bl\Theta_n$-isometries. 
\end{cor}

\begin{ex}
    Let $\mathfrak S_n$ denote the permutation group of $n$ symbols. The natural action of $ \mathfrak S_n$ on $\mb C^n$
is given as follows: For $(\sigma,\bl z)\in  \mathfrak S_n\times \mb C^n,$ $$(\sigma, \bl z)\mapsto {\sigma.\bl z}:=(z_{\sigma(1)},\ldots,z_{\sigma(n)}).$$ The symmetrization map \begin{eqnarray}\label{symm}\bl s :=(s_1,\ldots,s_n) : \mb D^n \to \mb G_n\end{eqnarray} is a basic polynomial map associated to the complex reflection group $\mathfrak S_n.$

A finite sequence $\bl p=(p_1\ldots p_k)$ of positive integers with $p_1\geq \ldots\geq p_k$ is  a \textit{partition} of $n,$ if $\sum_{i=1}^kp_i=n$, denoted by $\bl p \vdash n$. 
 The set  $\widehat{\mathfrak S_n}$ of equivalence classes of finite dimensional irreducible representations of the permutation group $\mathfrak S_n$ is parametrized by the partitions $\bl p$ of $n.$  
Since the character $\chi_{(n)}$ is identically $1$ on $\mathfrak S_n$, where $(n)$ denotes the partition $(n,0,\ldots, 0)$ that corresponds to trivial representation and the character $\chi_{(1^n)}$ is $1$ for an even permutation and $-1$ for an odd permutation in $\mathfrak S_n$, where $(1^n)$ denotes the partition $(1,\ldots, 1)$ that corresponds to sign representation, the permanent and the determinant of a matrix are special cases of immanants.

Let $\bl p \vdash n$ be a partition and $\mb{A}^{(\l)}_{\bl p, i} (\mb{D}^n) := \mb P^{ii}_{\bl p}(\mb{A}^{(\l)} (\mb{D}^n)).$ Then we have the following result from Theorem \ref{sm}.

\begin{thm}\label{Gamma}
    $\{(M_{s_1},\ldots,M_{s_n})|_{\mb P^{11}_{\bl p} (\mb A^{(\l)}(\mb D^n))} : \bl p \vdash n\}$ is a family of mutually unitarily inequivalent $\Gamma_n$-contractions.
\end{thm}

We provide an alternative proof of this theorem for one-dimensional representations of $\mathfrak{S}_n.$ We set $\mb{A}^{(\l)}_{\rm sign} (\mb{D}^n) := \mb P_{(1^n)}(\mb{A}^{(\l)} (\mb{D}^n))$ and $\mb{A}^{(\l)}_{\rm triv} (\mb{D}^n) := \mb P_{(n)}(\mb{A}^{(\l)} (\mb{D}^n)).$ 

\begin{prop}\label{inequiv}
The reducing submodules $\mb A^{(\l)}_{\rm triv}(\mb D^n)$ and $\mb A^{(\l)}_{\rm sign}(\mb D^n)$ over the ring $\mb C[z_1,\ldots,z_n]^{\mathfrak S_n}$ are not unitarily equivalent for $\l >1.$
\end{prop}
\begin{proof}

Fix a $\l>1.$ For $\bl w = (0,\ldots,0),$ we write $ K_{{\rm sign},0}(\cdot) := \mb P_{\rm sign}K_{\l}(\cdot, \bl w)\overline{\big(J_{\bl s}(\bl w)\big)^{-1}} = \frac{(\l)_{\bl\delta}}{\bl\delta! n!} J_{\bl s}(\cdot).$ For $\bl w=(0,\ldots,0),$ $\bigcap_{i=1}^n {\rm ker}(\mathbf M_{s_i}^{(\rm sign)} - s_i(\bl w))^*$ is spanned by $K_{{\rm sign},0}$ 
and $\bigcap_{i=1}^n {\rm ker}(\mathbf M_{s_i}^{(\rm triv)} - s_i(\bl w))^*$ is spanned by $1.$ 

Arguing by contradiction, suppose that $U : \mb A^{(\l)}_{\rm triv}(\mb D^n) \to \mb A^{(\l)}_{\rm sign}(\mb D^n)$ is a unitray map which intertwines the module actions of those Hilbert modules over the ring $\mb C[z_1,\ldots,z_n]^{\mathfrak S_n}.$ Then,  for every $\bl w \in \mb D^n$, the map $U^* : \bigcap_{i=1}^n {\rm ker}(\mathbf M_{s_i}^{(\rm sign)} - s_i(\bl w))^* \to \bigcap_{i=1}^n {\rm ker}(\mathbf M_{s_i}^{(\rm triv)} - s_i(\bl w))^* $ is unitary. Therefore, $U^*(K_{{\rm sign},0}) = \bar{c} \cdot 1,$ for a unimodular constant $c.$ (since $\big \Vert K_{{\rm sign},0}\big \Vert^2_{ \mb{A}^{(\l)} (\mb{D}^n)} = \frac{(\l)_{\bl\delta}}{\bl\delta!n!}
\big \Vert J_{\bl s}\big \Vert^2_{ \mb{A}^{(\l)} (\mb{D}^n)} = 1,$ $\ov{c}$ is unimodular.) Since $UU^* = I_{\rm sign},$ we conclude that $U(1)= c \cdot K_{{\rm sign},0}.$ Moreover, $U$ intertwines the module actions, so $U(p) = p U(1) = c \cdot (p K_{{\rm sign},0}) = c\frac{(\l)_{\bl\delta}}{\bl\delta!n!} \cdot (J_{\bl s}p),$ for $p \in \mb C[z_1,\ldots,z_n]^{\mathfrak S_n}.$ 

One can easily check that for $\bl m =(m_1,\ldots,m_n),$ $\Vert \bl z^{\bl m}\Vert^2_{ \mb{A}^{(\l)} (\mb{D}^n)} =\big \Vert  z_1^{ m_1}\ldots z_n^{m_n}\big \Vert^2_{ \mb{A}^{(\l)} (\mb{D}^n)}=\frac{m_1!\ldots m_n!}{(\l)_{m_1}\ldots (\l)_{m_n}}.$ Hence $\Vert s_n(\bl z) \Vert^2_{ \mb{A}^{(\l)} (\mb{D}^n)}= \Vert \bl z^{(1,\ldots,1)}\Vert^2_{ \mb{A}^{(\l)} (\mb{D}^n)} = \frac{1}{\l^n}.$

Again, $\Vert \bl z^{(1,\dots,1)} J_{\bl s}\Vert^2_{ \mb{A}^{(\l)} (\mb{D}^n)} = \frac{n! \bl m!}{(\l)_{m_1}\ldots (\l)_{m_n}},$ where $\bl m = (n,n-1,\ldots,2,1).$

Since $U$ is isometry, \Bea \Vert \bl z^{(1,\dots,1)} \Vert^2_{ \mb{A}^{(\l)} (\mb{D}^n)} &=& \frac{(\l)_{\bl\delta}}{\bl\delta!n!} \Vert \bl z^{(1,\dots,1)} J_{\bl s}\Vert^2_{ \mb{A}^{(\l)} (\mb{D}^n)}, \\ \frac{1}{\l^n} &=& \frac{(\l)_{n-1}(\l)_{n-2}\ldots (\l)_{1}(\l)_0}{\bl\delta!n!}  \frac{n! \bl m!}{(\l)_{n}(\l)_{n-1}\ldots (\l)_{2}(\l)_{1}}, \\ \frac{1}{\l^n} &=& \frac{n!}{(\l)_{n}}.\Eea
The last equality holds only if $\l =1,$ which leads us to a contradiction as $\l >1$.
\end{proof}
\end{ex}

\subsection{Similarity}
In view of Proposition \ref{inequiv}, it is natural to ask whether $\mb A_{\rm triv}^{(\l)}(\mb D^n)$ and $\mb A_{\rm sign}^{(\l)}(\mb D^n)$ are similar over $\mb C[z_1,\ldots,z_n]^{\mathfrak S_n}.$ We show for  $\l=1,2$ the submodules $\mb A_{\rm triv}^{(\l)}(\mb D^2)$ and $\mb A_{\rm sign}^{(\l)}(\mb D^2)$ are not similar over $\mb C[z_1,z_2]^{\mathfrak S_2}.$ We start with the following lemma.
\begin{lem}\label{surj}
Let $S:\mb A^{(\l)}_{\rm triv}(\mb D^2)\to \mb A^{(\l)}_{\rm sign}(\mb D^2)$ be a linear surjection satisfying $SM_{s_i}=M_{s_i}S$ for $i=1,2.$ Then $S=M_\varphi,$ where $\varphi=cJ_{\bl s}$ for some $c\in \mb C.$
\end{lem}
\begin{proof}
Let $S(1)=f$ for some $f\in\mb A^{(\l)}_{\rm sign}(\mb D^2).$  Then  $SM_{s_i}=M_{s_i}S$  for $i=1,2$ yields $Sg=gf$ for $g\in  \mb A^{(\l)}_{\rm triv}(\mb D^2).$   Since $S$ is a surjection there exists $h\in  \mb A^{(\l)}_{\rm triv}(\mb D^2)$ such that $Sh=J_{\bl s}=z_1-z_2.$ Thus  $Sh=hf=J_{\bl s}hf_{\rm triv}$ for $f_{\rm triv}\in \mb A^{(\l)}_{\rm triv}(\mb D^2).$ This forces $hf_{\rm triv}=1.$ This is impossible, unless $h$ and $f_{\rm triv}$ are both scalars. This completes the proof. 
\end{proof}
\begin{prop}\label{unbdd}
Let $S:\mb A^{(\l)}_{\rm triv}(\mb D^2)\to \mb A^{(\l)}_{\rm sign}(\mb D^2)$ be defined by $Sf=J_{\bl s}f.$ Then $S$ does not have a bounded inverse for $\l=1, 2.$
\end{prop}
\begin{proof}
Noting that $\mb C[z_1,z_2] \subseteq \mb A^{(\l)}(\mb D^2),$ consider the map $T:\mb P_{\rm sign} (\mb C[z_1,z_2])\to \mb C[z_1,z_2]^{\mathfrak S_2}$ defined  by $Tf=\frac{f}{J_{\bl s}}.$  If $T$ is bounded, then there exists $K>0$ satisfying $$\Vert Tf\Vert\leq K\Vert f\Vert \mbox{~for all~} f\in \mb P_{\rm sign}(\mb C[z]).$$
Recall that
\bea\label{norm}
\Vert z_1^mz_2^n\Vert^2=\frac{m!n!}{(\l)_m(\l)_n}.
\eea
Set $e_{m,0}(z_1,z_2)=z_1^m-z_2^m.$ Then $(Te_{m,0})(z_1,z_2)=\sum_{k=1}^{m}z_1^{m-k}z_2^{k-1}$ for $m\geq 1,$ and boundedness of $T$ implies that 
\bea\label{bdd}
\Vert \sum_{k=1}^{m}z_1^{m-k}z_2^{k-1}\Vert^2\leq K^2\Vert e_{m,0}\Vert^2.
\eea 
Putting $\l=2,$ 
it follows from Equation \eqref{bdd} that
$$
\sum_{k=1}^m\frac{1}{(m+1-k)k}\leq \frac{2 K^2}{m+1}.
$$
Equivalently, 
\bea \label{ineq}
\sum_{k=1}^m\frac{1}{(1-\frac{k}{m+1})k}\leq 2K^2 \mbox{~for~} m\geq 1.
\eea
Note that $1-\frac{k}{m+1}<1.$ Hence $\sum_{k=1}^m\frac{1}{(1-\frac{k}{m+1})k}\geq \sum_{k=1}^m\frac{1}{k}.$ Thus, Equation \eqref{ineq} cannot hold for all $m\geq 1.$

To prove the assertion for $\l=1,$ we simply  note from Equation \eqref{norm} that  Equation \eqref{bdd} reduces to $m^2\leq 4K^2$ for all $m\geq 1.$ This contradiction completes the proof.
\end{proof}



\begin{rem}\label{div}
In general, the map $f\mapsto V_n f$, where $V_n(z) = \prod_{i<j}(z_i - z_j)$, from $\mathbb A^{(\lambda)}_{triv}(\mathbb D^n)$ to $\mathbb A^{(\lambda)}_{sign}(\mathbb D^n)$ is one-one and one ask whether this map is onto or not to have a similarity. It turns out that this is the same is asking whether $V_n$ is a 'divisor' in the sense of \cite{D, PS} - that is, if $f\in \mathbb A^{(\lambda)}(\mathbb D^n)$ and $f/V_n\in \mathcal O(\mathbb D^n)$, where $\mathcal O(\mathbb D^n)$ is the algebra of holomorphic functions on $\mathbb D^n$, then is it true that $f/V_n \in \mathbb A^{(\lambda)}(\mathbb D^n)$? The Proposition above shows that the division problem fails for $\lambda = 1, 2$ when $n = 2$.
\end{rem}

\begin{cor}\label{simlr}
There is no linear surjection $S:\mb A^{(\l)}_{\rm triv}(\mb D^2)\to \mb A^{(\l)}_{\rm sign}(\mb D^2)$ satisfying $SM_{s_i}=M_{s_i}S$ for $i=1,2$ and $\l=1,2.$ Consequently, the submodules $\mb A^{(\l)}_{\rm triv}(\mb D^2)$ and $ \mb A^{(\l)}_{\rm sign}(\mb D^2)$ are not similar for $\l=1,2.$
 \end{cor}
\begin{proof}
Follows immediately from Lemma \ref{surj}, Proposition \ref{unbdd} and the open mapping theorem.
 \end{proof}   
 
The previous corollary can be rephrased in terms of $\Gamma_2$-contractions.
\begin{thm}\label{similar}
    The $\Gamma_2$-contractions $(M_{s_1},M_{s_2})|_{\mb A^{(\l)}_{\rm triv}(\mb D^2)}$ and $(M_{s_1},M_{s_2})|_{\mb A^{(\l)}_{\rm sign}(\mb D^2)}$ are not similar for $\l =1,2$. 
\end{thm}

\begin{ex}
    Consider the polynomial map $$\bl \phi(z_1,z_2):=(\phi_1,\phi_2) = (z_1^k + z_2^k, z_1z_2)$$ on $\mb D^2.$ It is a proper holomorphic map of multiplicity $2k$ and let us denote the domain $\phi(\mb D^2)$ by $\mathscr{D}_{2k}.$ The domain $\mathscr{D}_{2k}$ is biholomorphically equivalent to the quotient domain $\mb{D}^2/D_{2k}$  where $D_{2k}$ is the dihedral group of order $2k,$ that is,
$$D_{2k} = \langle \delta, \sigma : \delta^k=\sigma^2 ={\rm id, \sigma \delta \sigma^{-1} = \delta^{-1}} \rangle.$$
See \cite[subsection~3.1.1]{BDGS} and \cite[p. 19]{Ghosh} for more details. Clearly, $J_{\bl \phi}(z_1,z_2) = k(z_1^k -z_2^k).$

We get the following result from Theorem \ref{sm}. For some $\mu \in \widehat{D}_{2k},$ set $\mb A^{(\l)}_{\mu,i} (\mb D^2):=\mb P^{ii}_\mu (\mb A^{(\l)}(\mb D^2)).$ Let $\mu_1$ and $\mu_2$ be two inequivalent representations of $D_{2k}.$ The joint reducing submodules $\mb A^{(\l)}_{\mu_1,i}(\mb D^2)$ and $\mb A^{(\l)}_{\mu_2,j}(\mb D^2)$ over the ring $\mb C[z_1,z_2]^{D_{2k}}$ are not unitarily equivalent for $\l \geq 1.$
\begin{thm}\label{Di}
    $\{(M_{\phi_1},M_{\phi_2})|_{\mb P^{11}_\mu (\mb A^{(\l)}(\mb D^2))} : \mu\in \widehat{D}_{2k}\}$ is a family of mutually (unitary) inequivalent $\mathscr D_{2k}$-contractions.
\end{thm}

The number of one-dimensional representations of the dihedral group $D_{2k}$ in $\widehat{D}_{2k}$ is $2$ if $k$ is odd and $4$ if $k$ is even. Clearly, for every $k \in \mb N$ the trivial representation of $D_{2k}$ and the sign representation of $D_{2k}$ are in $\widehat{D}_{2k}.$ 

\begin{itemize}
    \item The trivial representation ${\rm triv}:D_{2k} \to \mb C^*$ is given by ${\rm triv}(\sigma) =1$ for $\sigma \in D_{2k}.$ 
    \item The sign representation is as defined in Equation \eqref{sign}. 

    \item 
The additional two representations in the case when $k$ is even.  If $k=2j$ for some $j\in \mb N.$ We consider the representation $\varrho_1$ defined as 
\begin{eqnarray*} \varrho_1(\delta) = -1 &\text{~and~}& \varrho_1(\tau) =1 \text{~for~} \tau \in \inner{\delta^2}{\sigma}.
\end{eqnarray*}
\item
For $k=2j,$ the representation $\varrho_2$ is defined as following:
\begin{eqnarray*}
\varrho_2(\delta) = -1 &\text{~and~}& \varrho_2(\tau) =1 \text{~for~} \tau \in \inner{\delta^2}{\delta\sigma}.\end{eqnarray*}

\end{itemize}

Every pair of representations described above are mutually inequivalent. This leads to the following corollary of Theorem \ref{Di}.
\begin{cor}\label{Di1}
    Let $\l\geq 1$ and $\mu_1$ and $\mu_2$ be two inequivalent one-dimensional representations of $D_{2k}.$ The $\mathscr D_{2k}$-contractions $(M_{\phi_1},M_{\phi_2})|_{\mb P_{\mu_1}\big(\mb A^{(\l)}(\mb D^2)\big)}$ and $(M_{\phi_1},M_{\phi_2})|_{\mb P_{\mu_2}\big(\mb A^{(\l)}(\mb D^2)\big)}$ are unitraily inequivalent.
\end{cor}
Following analogous arguments as in the case of the group $G(1,1,n)$, it can be shown that the  $\mathscr D_{2k}$-contractions on the submodules of weighted Bergman module $\mb A^{(\l)}(\mb D^2), ,\,\ \l\geq 1,$ associated  to the one-dimensional representations of $G(k,k,2)$ are not mutually similar. 
\end{ex}

\end{document}